\documentclass[reqno,a4paper]{amsart}

\usepackage[all,numberbysection]{preprint_style}

\usepackage[backend=biber,maxbibnames=5,maxalphanames=5,style=alphabetic,bibencoding=utf8,giveninits,url=false,isbn=false,doi=true,eprint=false]{biblatex}
\renewbibmacro{in:}{}
\AtBeginBibliography{\small}

\usepackage[T1]{fontenc}
\allowdisplaybreaks

\usepackage[top=2.75cm,bottom=2.75cm,left=3cm,right=3cm]{geometry}

\usepackage{upgreek}
\usepackage{stmaryrd}
\usepackage{amsthm}
\usepackage{mathtools}
\usepackage{thmtools}
\usepackage{mathrsfs}

\providecommand{\divo}{\textup{div}\,}
\newcommand{\Divhk}{\mathcal{D}\dot{\iota}\nu_h^k}
\newcommand{\Divhke}{\mathcal{D}\dot{\iota}\nu_h^{k-1}}

\newcommand{\Dhke}{\boldsymbol{\mathcal{D}}_h^{k-1}}
\newcommand{\Dhk}{\boldsymbol{\mathcal{D}}_{h}^k}

\newcommand{\Rhk}{\boldsymbol{\mathcal{R}}_h^k}
\newcommand{\Rhks}{\boldsymbol{\mathcal{R}}_h^{\smash{k,{\textup{sym}}}}}
\newcommand{\Rhke}{\boldsymbol{\mathcal{R}}_h^{k-1}}
\newcommand{\Ghk}{\boldsymbol{\mathcal{G}}_h^k}
\newcommand{\Ghke}{\boldsymbol{\mathcal{G}}_h^{k-1}}
\providecommand{\PiDG}{{\Uppi_{h}^{k}}}
\providecommand{\PiDGe}{{\widetilde{\boldsymbol{\Uppi}}_{h}^{0}}}
\providecommand{\vep}{\varepsilon}
\providecommand{\SSS}{\BS}

\providecommand{\Vhk}{\smash{V_h^k}}
\providecommand{\Vhke}{\smash{V_h^{k-1}}}
\providecommand{\Vo}{\mathaccent23 V}
\providecommand{\Qo}{\mathaccent23 Q}
\providecommand{\Qhk}{\smash{Q_h^k}}
\providecommand{\Qhko}{{\mathaccent23 Q}_h^{k}}
\providecommand{\Qhke}{\smash{Q_h^{k-1}}}
\providecommand{\Qhkeo}{{\mathaccent23 Q}_h^{k-1}}
\newcommand{\WDG}{W^{1,p}(\mathcal{T}_h)}
\newcommand{\WDGd}{W^{1,\psi}_{\divo}(\mathcal{T}_h)}
\providecommand{\Xhk}{\smash{X_h^k}}
\providecommand{\Xhks}{\smash{X_{h}^{k-1,\textup{sym}}}}
\providecommand{\tria}{\triang}
\providecommand{\triaMK}{\mathcal{M}_K}
\newcommand\leftavg{\{\!\!\{}
\newcommand\rightavg{\}\!\!\}}

\providecommand{\fdg}{{\,\big|\,}}

\newtheorem{lem}[equation]{Lemma}
\newtheorem{thm}[equation]{Theorem}
\newtheorem{prop}[equation]{Proposition}
\newtheorem{ass}[equation]{Assumption}
\newtheorem{rem}[equation]{Remark}
\newtheorem{cor}[equation]{Corollary}

\begin{document}
	
	\title[Pressure-robust DG $p$-Stokes]{Pressure-robust and quasioptimal Discontinuous Galerkin discretisations of the $p$-Stokes problem}
	
	\author[P.A.~Gazca-Orozco]{P.A.\ Gazca-Orozco}
	\author[M.~\Ruzicka]{M.\ \Ruzicka}
	
	\address[P.A. Gazca-Orozco]{Faculty of Mathematics and Physics, Charles University, Sokolovská 83, 186 75, Prague, Czech Republic}
	\email{gazca@karlin.mff.cuni.cz}
\address[M. \Ruzicka]{Department of Applied Mathematics, University of Freiburg, 79104, Freiburg, Germany}
\email{rose@mathematik.uni-freiburg.de}
	
	
	
	\date{\today}
	
	\setcounter{tocdepth}{1} 	
	\begin{abstract}
  In the present paper, we propose Local Discontinuous Galerkin (LDG)
  approximations for a nonlinear system of $p$-Stokes type, having
  $(p,\delta)$-structure.  On the basis~of~the~primal~formulation, we
  prove well-posedness, and stability (\textit{a priori} estimates) of
  the methods under truly minimal regularity assumptions. We show that
  the first method possesses a pressure-robust and quasi-optimal error
  estimate, and discuss its consequences. Moreover, we propose a
  second method, for which we show a pressure-robust error estimate
  and prove convergence and convergence rates, which are optimal for
  linear ansatz functions for all $p\in (1,\infty)$ and $\delta\ge 0$.
	\end{abstract}

    \maketitle


  \section{Introduction}
In this paper, we examine Local Discontinuous Galerkin (LDG) 
discretisations of non-linear problems of \textit{$p$-Stokes
type},~i.e.,
\begin{equation}
  \label{eq:p-stokes}
  \begin{aligned}
    -\divo\SSS(\BD\bv)+\nabla q&=\bff \qquad&&\text{in }\Omega\,,\\
    \divo\bv&=0 \qquad&&\text{in }\Omega\,,
    \\
    \bv &= \bzero &&\text{on } \partial\Omega\,.
  \end{aligned}
\end{equation}
The physical problem motivating this study is the laminar, steady
motion of a homogeneous, incompressible fluid with shear-dependent
viscosity. More precisely, for a given external body force
$\bff\colon\Omega\to\mathbb{R}^d$, we seek a
velocity field 
${\bv\hspace{-0.1em}=\hspace{-0.1em}(v_1,\dots,v_d)^\top\hspace{-0.1em}\colon\hspace{-0.1em}\Omega\hspace{-0.1em}\to\hspace{-0.1em}
  \mathbb{R}^d}$ and a
kinematic pressure~${q\hspace{-0.1em}\colon\hspace{-0.1em}\Omega\hspace{-0.1em}\to\hspace{-0.1em}
  \mathbb{R}}$ {solving \eqref{eq:p-stokes}.  Here,
$\Omega\hspace{-0.1em}\subseteq\hspace{-0.1em} \mathbb{R}^d$,
$d\hspace{-0.1em}\in\hspace{-0.1em} \{2,3\}$, is a bounded polyhedral
domain having~a~Lipschitz \mbox{continuous} boundary
$\partial\Omega$. The extra stress tensor
$\SSS \hspace{-0.1em}\colon\hspace{-0.1em}\mathbb{R}^{d\times d}\hspace{-0.1em}\to\hspace{-0.1em}\mathbb{R}^{d\times d}_{\textup{sym}}$ depends on the~strain~rate~tensor
$\BD\bv\coloneqq \frac{1}{2}(\nabla\bv + \nabla\bv^\top)$,
  i.e., the symmetric
  part of the velocity gradient
${\nabla \bv 
\colon  \Omega\to
\mathbb{R}^{d\times d}}$.  Physical interpretation and discussion of
some non-Newtonian fluid models can be found, e.g., in
\cite{bird,mrr,ma-ra-model}.

Throughout the paper, we assume that the extra stress
tensor~$\SSS$~has~\mbox{$(p,\delta)$-structure}
(cf.~Ass.~\ref{assum:extra_stress}). The prototypical example
falling into this class is
\begin{align*}
    \SSS(\BD\bv)=\mu\, (\delta+\vert \BD\bv\vert)^{p-2}\BD\bv\,,
\end{align*}
where $p\in (1,\infty)$, $\delta\ge 0$, and $\mu>0$.
The mathematical investigation of fluids with shear-dependent viscosities
started with the celebrated work of O.\ Ladyzhenskaya
(cf.~\cite{lady-bo}). In recent years, there has been enormous
progress in the understanding of this problem, and we refer
the~reader~to \cite{mnrr,ma-ra-model,fms2,dms,die-ru-wolf,hugo-boundary,hugo-petr-rose,br-reg-shearthin,bdr-phi-stokes,die-sueli-2013,kr-pnse-ldg-1} 
and the references therein for a detailed discussion.

The objective of this work is to propose a Local Discontinuous
Galerkin (LDG) scheme for \eqref{eq:p-stokes} for which one can
establish a \textit{minimal regularity, quasi-optimal, pressure-robust} (a priori) error
estimate, i.e., a best-approximation result of the form
\begin{align}\label{eq:error_estimate}
  \begin{aligned}
    &\|\BF(\Dhke \bv_h) - \BF(\BD \bv)\|^2_{2,\Omega} +
    m_{\varphi_{\bbeta_h(\bv_h)},h}(\bv_h)
    \\
    &\lesssim \inf_{\bw_h \in V_h^k(0)} \left(\|\BF(\Dhke \bw_h) - \BF(\BD
      \bv)\|^2_{2,\Omega} +
      m_{\varphi_{\bbeta_h(\bv_h)},h}(\bw_h)\right)\,,
  \end{aligned}
\end{align}
where $\bv_h\in V_h^k(0)$ is the discrete solution,
$\Dhke\colon W^{1,p}(\mathcal{T}_h)^d\to L^p(\Omega)^{d\times d}$ is
the symmetric DG gradient, and $\BF$ and
$m_{\varphi_{\bbeta_h(\bv_h)},h}$ are appropriate measures of the
error related to the $(p,\delta)$-structure~of~$\BS$ (cf.~\eqref{eq:def_F}, \eqref{eq:mod1}, \eqref{def_shift}). Here, $V_h^k$ is
typically a space of broken polynomials, i.e.,
$\smash{V_h^k \coloneqq \mathbb{P}^k(\triang)^d}$, on
a~\mbox{triangulation}~$\triang$~of~$\Omega$, which is meant to
approximate the full space $W^{1,p}(\Omega)^d$, and $\Vhk(0)$ is the
subspace of $\Vhk$ of vector fields $\bu_h$ with $\Divhk \bu_h=0$,
i.e., the DG divergence is zero. The estimate
\eqref{eq:error_estimate} is called quasi-optimal and pressure-robust,
because the velocity error is proportional to the best approximation
error measured in the error measure related to the
$(p,\delta)$-structure~of~$\BS$ and is independent of the pressure.
Crucially, when deriving the estimate \eqref{eq:error_estimate} we
only assume the natural regularity of the continuous problem; namely,
$\smash{\bu\in W^{1,p}_{0,\divo}(\Omega)^d}$, the subspace of
solenoidal vector fields from $\smash{ W^{1,p}_{0}(\Omega)^d}$, and
$\smash{\bff\in W^{-1,p'}(\Omega)^d}$. 

Pressure-robustness of the estimate in \eqref{eq:error_estimate} refers to the fact that the convergence of the velocity field is not affected by the quality of the pressure approximation.
This is a topic that has received much attention in recent years,
mostly pertaining to the linear Stokes problem.
It is a very desirable property of discrete approximations of incompressible flow,
since non-robust discretisations may produce larger errors for small viscosities, and even unphysical behaviour;
for more details see e.g.\ \cite{Lin.2014,JLMNR17,LLMS.2017,KZ.2020,KVZ.2021} (and the references therein).
Moreover, as elaborated in \cite{DHKZ.2025}, non-pressure-robust discretisations of non-Newtonian flow lead in general to suboptimal rates for the velocity.
In the quasilinear setting there are fewer results in this direction available; classical results for conforming \cite{barliu,bdr-phi-stokes,Hi13a,kr-pnse-ldg-2} and non-conforming \cite{bustinza,CHSW13,GS15,BCPH20,CdPH.2023} discretisations are not pressure-robust.
Regarding robust estimates we can mention for example those for a two-dimensional high-order approximation from \cite{PS.2025} or the $H(\diver;\Omega)$-conforming method from \cite{dVdPH} (similar results might be obtained from the DG discretisations mentioned earlier by taking $H(\diver)$-conforming subspaces).
However, all of these results hold assuming some extra regularity of the exact solution: they are not \emph{minimal regularity quasioptimal} estimates.
To date, the only such results for a non-conforming scheme were derived in \cite{DHKZ.2025} for a Crouzeix--Raviart approximation.

A whole theory dealing with the characterisation of truly
quasi-optimal discretisations of symmetric and elliptic linear
problems in $\smash{W^{1,2}_0(\Omega)^d}$ was developed in
\cite{VZ.2018.I,VZ.2019.II,VZ.2018.III}. These results have been
extended to the linear Stokes problem in \cite{VZ.2019,KZ.2020,KVZ.2021},
showing quasi-optimal and even pressure-robust error estimates.  In
\cite{bgkr-optimal-dg} a quasi-optimal DG discretisation for problems
of $p$-Laplacian type is proposed, which implies convergence of
the method under minimal regularity assumptions, and convergence rates
under additional regularity conditions on the velocity, which are
optimal for linear ansatz functions.
Instrumental in the derivation of these results is the presence of a so-called smoothing operator $\BE_h$ in the forcing term: $\langle \bm{f},\BE_h\bw_h \rangle_{W^{-1,p'}(\Omega);W^{1,p}_0(\Omega)}$;
this operator maps the discrete space $V^k_h$ into $W^{1,p}_0(\Omega)$-conforming functions, and satisfies certain structural properties (see Section \ref{sec:smoothing});
for $p\not=2$ and quadratic ansatz functions (or higher), this operator needs to be applied to all the test functions in the discrete formulation (see Rem.~\ref{rem:grad_tilde}).
The only work with minimal regularity and pressure-robust estimates \cite{DHKZ.2025} mentioned previously,
proposes two methods: the first one requires the implementation of $\BE_h$ in the right-hand-side only (just like all the previous results for scalar and/or linear problems), but considers a nonlinear term $\BS(\nabla \bv)$ involving the full gradient, which is unphysical.
The second method requires the application of $\BE_h$ to all \emph{trial and test functions},
which is equivalent  to a conforming and divergence-free discretisation on the image of the
smoothing operator $\BE_h $ applied to the Crouzeix--Raviart finite
element space and thus, pressure-robustness follows easily (and no jump penalisation is required).
This work strives on the other hand to follow the approach from previous works, in which $\BE_h$ is applied to test functions only,
while at the same time involving the symmetric velocity gradient.
Our results are thus complementary to those from \cite{DHKZ.2025}, and cover also other DG discretisations.


Traditionally, a quasioptimality bound like \eqref{eq:error_estimate} is enough to conclude plain convergence of the scheme for minimal regularity solutions, and convergence rates with additional regularity of the velocity (one expects e.g.\ linear convergence for linear ansatz functions).
Surprisingly, however, we ran into a couple of issues that prevented us in carrying out this argument in the case $p>2$,
but also partially for $p\leq 2$.
The problem arises due to the non-conformity (jump) term $m_{\varphi_{\beta_h}(\bv_h),h}(\bw_h)$;
in the results for the $p$-Laplace problem from \cite{bgkr-optimal-dg}, it was possible to pick $\bw_h$ as a quasi-interpolant of $\bv$ into a conforming subspace,
so this term vanished completely. 
In contrast, the estimate \eqref{eq:error_estimate} would require $\bw_h$ to be discretely divergence-free as well,
and to the best of our knowledge there is currently no quasi-interpolation operator with this property.
For this reason we resort to setting $\bw_h$ as the (local) $L^2$-projection of $\bv$ onto $V^k_h(0)$,
which enables us to prove for $p\leq 2$ that the right-hand-side of \eqref{eq:error_estimate} (and hence the \emph{squared} error) decays to zero for minimal regularity solutions,
and with a rate $h^p$ for linear elements with additional regularity (and $\delta>0$), which is suboptimal compared to the optimal $h^2$ rate.

 Thus, inspired by the recent works \cite{kr-orlicz-ldg,kr-pnse-ldg-1,kr-pnse-ldg-2},
 we propose also an alternative LDG scheme for \eqref{eq:p-stokes} with a modified non-conformity penalisation,
 for which we derive a minimal regularity, pressure-robust, \emph{weak} quasi-optimality bound:
 the velocity error is bounded by the best approximation error plus additional terms, but is still independent of the pressure.
We term this weak quasioptimality, since even though the best-approximation error is not necessarily proportional to the velocity error,
the asymptotic behaviour is still optimal: as a consequence we obtain plain convergence for minimal regularity solutions,
and additionally linear convergence for linear ansatz functions (under the usual regularity assumptions),
with no restrictions on $p$ or $\delta$. 

One of the takeaways from this work is that care should be taken when defining concepts such as quasioptimality for nonlinear problems;
for $p$-Stokes-type problems it is still not clear what is the correct way of measuring and penalising the non-conformity of the approximations.
We showed that the jump penalisation that worked for the scalar problem in \cite{bgkr-optimal-dg} led to quasioptimality,
but made the derivation of optimal error rates problematic,
whereas an alternative penalisation led to optimal rates but no genuine quasioptimality (note that both of these reduce to the usual DG jump penalisation for $p=2$).

\textit{This paper is organized as follows:} In Section
\ref{sec:preliminaries}, we introduce the employed notation,
define relevant function spaces, basic assumptions on the extra
stress~tensor~$\SSS$~and~its consequences, weak formulations in
Problem (Q) and Problem (P) of the~system~\eqref{eq:p-stokes}, and
discrete operators.  In Section \ref{sec:ldg}, we propose a primal
formulation, i.e, Problem (Q$_h$) and Problem (P$_h$) of the system
\eqref{eq:p-stokes}. Moreover, we prove that this LDG scheme allows for
a quasi-optimal and pressure-robust error estimate, and discuss its consequences.  In Section
\ref{sec:rob}, we propose a modified scheme and prove a
pressure-robust error estimate, which allows us to prove convergence
and convergence rates for all $p\in (1,\infty)$ and $\delta\ge 0$. 

\section{Preliminaries}\label{sec:preliminaries}

\subsection{Function spaces}

We employ $c,  C$ to denote generic constants, that may change from
line to line and may depend only on the polynomial degree $k$, the
chunkiness $\omega_0$, the characteristics of $\BS$, and the dimension
$d$. Moreover, we~write $f \lesssim g$ if there exist a constant $c>0$
such that $f \le c\, g$, and ${f\sim g}$ if and only if there exists
constants $c, C>0$ such that $c\, f \le g\le C\, f$.

Throughout the paper, let $\Omega\subseteq \mathbb{R}^d$,
$d \in \{2,3\}$, be a bounded, polyhedral Lipschitz domain~and
$M\subseteq \mathbb{R}^d$, $d \in \{2,3\}$, a (Lebesgue) measurable
set. Then, for every $k\in \setN$ and $p\in [1,\infty]$, we employ the
customary Lebesgue spaces $(L^p(M), \smash{\|\cdot\|_{p,M}}) $ and
Sobolev spaces $(W^{k,p}(M), \smash{\|\cdot\|_{k,p,M}})$. The space
$W^{1,p}_0(\Omega)$ is defined as those functions from
$W^{1,p}(\Omega)$ whose trace vanishes on $\partial\Omega$. We equip
$\smash{W^{1,p}_0(\Omega)}$ with the norm
$\smash{\norm{\nabla\,\cdot\,}_p}$.

We always denote
vector-valued functions by boldface letters~and~tensor-valued
functions by capital boldface letters. The standard scalar product
between~two vectors is denoted by $\ba \cdot\bb$, while the
  Frobenius scalar product between~two~tensors is denoted by
  $\BA: \BB$.  The mean value of a locally integrable function $f$
  over a measurable set $M\subseteq \Omega$ is denoted by
  ${\mean{f}_M\coloneqq \smash{\dashint_M f
      \,\textup{d}x}\coloneqq \smash{\frac 1 {|M|}\int_M f
      \,\textup{d}x}}$. Moreover, we employ the notation
  $\hskp{f}{g}_\Omega\coloneqq \int_\Omega f g\,\textup{d}x$, whenever the
  right-hand~side~is~\mbox{well-defined}. For a Banach space $V$ we
  denote by $V^*$ its dual space and by $\langle\cdot, \cdot\rangle
  _V$ the duality pairing.

From the theory of Orlicz spaces $L^\psi (M)$ and
Sobolev--Orlicz spaces $W^{\ell,\psi} (M)$, $\ell \in \setN$, 
(cf.~\cite{ren-rao}) and generalized
Orlicz~spaces~$L^{\psi(\cdot)} (M)$ (cf.~\cite{HH19}), we employ
\mbox{N-functions} $\psi \colon \setR^{\geq 0} \to \setR^{\geq 0}$ and
generalized \mbox{N-functions}
$\psi \colon M \times \setR^{\ge 0} \to \setR^{\ge 0}$, i.e., $\psi$
is a Carath\'eodory function such that $\psi(x,\cdot)$ is an
N-function for a.e.~${x \in M}$,~\mbox{respectively}. The
{\it modular}~is~defined~via
$\rho_{\psi,M}(f)\coloneqq \int_M \psi(\abs{f})\,\textup{d}x $ if
$\psi$ is an N-function, and via
$ \rho_{\psi,M}(f)\coloneqq \int_M \psi(x,\abs{f(x)})\,\textup{d}x $,
if $\psi$ is a generalized N-function. An N-function $\psi$ satisfies
the $\Delta_2$-condition (in short, $\psi \in \Delta_2$), if there
exists $K> 2$ such that for every
$t \ge 0$,~it~holds~that~${\psi(2\,t) \leq K\, \psi(t)}$. We denote
the smallest such constant by ${\Delta_2(\psi)>2}$. We define the
(convex) conjugate (generalized) N-function
${\psi^*\colon M\times \mathbb{R}^{\ge 0}\to \mathbb{R}^{\ge 0}}$ via \linebreak
$\psi^*(x,t)\coloneqq \sup_{s\ge 0}{ ts-\psi(x,s)}$ for all $t\ge 0$
and a.e.\ ${x\in M}$. If $\psi, \psi^* \in \Delta_2$, then we have
that
\begin{align}
  \label{eq:psi'}
  \psi^* \circ \psi'\sim \psi\,,
\end{align}
with constants depending only on $\Delta_2(\psi),\Delta_2( \psi ^*)$.
We will also need the $\varepsilon$-Young inequality: for every
$\varepsilon\hspace*{-0.1em}> \hspace*{-0.1em} 0$, there exits a
constant $c_\varepsilon\hspace*{-0.1em}>\hspace*{-0.1em}0 $, depending
only on $\Delta_2(\psi),\Delta_2( \psi ^*)$, such that for every
$s,t\geq 0$, it holds that
\begin{align}
  \label{ineq:young}
  \begin{aligned}
    t\,s&\leq \varepsilon \, \psi(t)+ c_\varepsilon \,\psi^*(s)\,.
  \end{aligned}
\end{align}

\subsection{Basic properties of the extra stress tensor}

Throughout the paper, we assume that the extra stress tensor $\SSS$
has $(p,\delta)$-structure. A detailed discussion and full proofs can
be found in \cite{die-ett,dr-nafsa}. For a tensor
$\BA\in \mathbb{R}^{d\times d}$, we denote its symmetric part by
$\BA^{\mathrm{sym}}\coloneqq \frac{1}{2}(\BA+\BA^\top)\in \RRddsym\coloneqq \{\BA\in \RRdd \mid \BA=\BA^\top\}$.

For $p \in (1,\infty)$ and~$\delta\ge 0$, we define a special N-function
$\phi=\phi_{p,\delta}\colon\smash{\mathbb{R}^{\ge 0}\to \mathbb{R}^{\ge 0}}$~by
\begin{align} 
  \label{eq:5} 
  \varphi(t)\coloneqq  \int _0^t \varphi'(s)\, \mathrm ds\,,\quad\text{where}\quad
  \varphi'(t) \coloneqq  (\delta +t)^{p-2} t\quad\textup{ for all }t\ge 0\,.
\end{align}
It is well-known that $\phi$ is balanced
(cf.~\cite{dr-nafsa,br-multiple-approx}), since
$ {\min\set{1,p-1}\,( \delta+t)^{p-2} \le \varphi''(t)}$
$\leq \max\set{1,p-1}( \delta+t)^{p-2}$ for~all~${t,\delta\ge
  0}$.

An important tool in our analysis are {\it shifted N-functions}
$\set{\psi_a}_{\smash{a \ge 0}}$ (cf.~\cite{DK08,dr-nafsa}). For a
given N-function $\psi\colon\mathbb{R}^{\ge 0}\to \mathbb{R}^{\ge 0}$,
we define the family of shifted N-functions
${\psi_a\colon\mathbb{R}^{\ge 0}\to \mathbb{R}^{\ge 0}}$ for every
$a\ge 0$, via
\begin{align}
  \label{eq:phi_shifted}
  \psi_a(t)\coloneqq  \int _0^t \psi_a'(s)\, \mathrm ds\,,\quad\text{where }\quad
  \psi'_a(t)\coloneqq \psi'(a+t)\frac {t}{a+t}\,,\quad\textup{ for all }t\ge 0\,.
\end{align}

\begin{rem} \label{rem:phi_a} {\rm For the above defined special N-function
    $\varphi $ we~have, uniformly in $a,t\!\ge\! 0$, that
    $\phi_a(t)\! \sim\! (\delta+a+t)^{p-2} t^2$~and
    $(\phi_a)^*(t) \!\sim\! ((\delta+a)^{p-1} + t)^{p'-2} t^2$.  The
    families $\set{\phi_a}_{\smash{a \ge 0}}$~and
    $\set{(\phi_a)^*}_{\smash{a \ge 0}}$ satisfy, uniformly in
    $a,\delta \ge 0$, the $\Delta_2$-condition with
    ${\Delta_2(\phi_a) \!\lesssim \!2^{\smash{\max \set{2,p}}}}$ and
    $\Delta_2((\phi_a)^*) \!\lesssim\! 2^{\smash{\max \set{2,p'}}}$,
    respectively. Moreover, note that
    $(\phi_{p,\delta})_a(t)=\phi_{p, \delta+a}(t)$ for all
    $t,a,\delta\ge 0$, and that for all $0\le a\le b$, we have that
    for every $t \ge0$, it holds that
    $(\varphi_a)^*(t) \ge (\varphi_b)^*(t) $,
    $\varphi_a(t) \le \varphi_b(t) $~if~$p\ge 2$ and
    $(\varphi_a)^*(t) \le (\varphi_b)^*(t) $,
    $\varphi_a(t) \ge \varphi_b(t) $ if $p\le 2$. }
  \end{rem}

%
\begin{ass}[Extra stress tensor]\label{assum:extra_stress}
  We assume that the extra stress tensor 
  $ \SSS\hspace{-0.15em}\colon\hspace{-0.15em} \mathbb{R}^{d \times d}
  \hspace{-0.15em}\to\hspace{-0.15em} \mathbb{R}^{d \times d}_{\textup{sym}} $ belongs to $C^0(\mathbb{R}^{d \times
    d}\hspace{-0.1em},\mathbb{R}^{d \times d}_{\textup{sym}} ) $, satisfies $\SSS (\hspace{-0.05em}\BA\hspace{-0.05em}) \hspace{-0.15em}=\hspace{-0.15em} \SSS 
  (\hspace{-0.05em}\BA^{\textup{sym}}\hspace{-0.05em} )$~for~all~${\BA\hspace{-0.15em}\in\hspace{-0.15em} \mathbb{R}^{d \times d}}\hspace{-0.1em}$, and $\SSS (\mathbf 0)=\mathbf 0$. Furthermore, we
  assume that the tensor $\SSS$ has {\rm $(p, \delta)$-structure}, i.e.,
  for some $p \in (1, \infty)$, $ \delta\in [0,\infty)$, and the
  N-function $\varphi=\varphi_{p,\delta}$ (cf.~\eqref{eq:5}), there
  exist constants $C_0, C_1 >0$ such that
   \begin{equation}
     \label{eq:ass_S}
     \begin{aligned}
       \big({\SSS}(\BA) - {\SSS}(\BB)\big) : \big(\BA-\BB
       \big) &\ge C_0 \,\phi_{\abs{\BA^{\textup{sym}}}}(\abs{\BA^{\textup{sym}} -
         \BB^{\textup{sym}}}) \,,
       \\
       \abs{\SSS(\BA) - \SSS(\BB)} &\le C_1 \,
       \phi'_{\abs{\BA^{\textup{sym}}}}\big(\abs{\BA^{\textup{sym}} -
         \BB^{\textup{sym}}}\big)
     \end{aligned}
   \end{equation}
   are satisfied for all $\BA,\BB \in \mathbb{R}^{d \times d} $.  The
   constants $C_0,C_1>0$ and $p\in (1,\infty)$ are called the 
     {\em characteristics of $\SSS$}.
\end{ass}

\begin{rem}
  {\rm (i) Let $\phi$ be defined in \eqref{eq:5} and let
    $\{\phi_a\}_{a\ge 0}$ be the corresponding family of the~shifted \mbox{N-functions}. Then, the operators 
    $\SSS_a\colon\mathbb{R}^{d\times d}\to \smash{\mathbb{R}_{\textup{sym}}^{d\times
      d}}$, $a \ge 0$, defined, for every $a \ge 0$
    and~$\BA \in \mathbb{R}^{d\times d}$, via
\begin{align}
  \label{eq:flux}
  \SSS_a(\BA) \coloneqq 
  \frac{\phi_a'(\abs{\BA^{\textup{sym}}})}{\abs{\BA^{\textup{sym}}}}\,
  \BA^{\textup{sym}}\,, 
\end{align}
have $(p, \delta +a)$-structure (cf.~Rem.~\ref{rem:phi_a}).  In this case, the characteristics of
$\SSS_a$ depend only on $p\in (1,\infty)$ and are independent of
$\delta \geq 0$ and $a\ge 0$.

{(ii) Note that the $(p,\delta)$-structure contains as particular cases, e.g.,~the Ladyzhenskaya model, the Smagorinski model, power law models, the
Carreau--Yasuda model and the  Powell--Eyring model. 
  }}
\end{rem}

Closely related to the extra stress tensor $\SSS$ with
$(p,\delta)$-structure are the functions
${\BF\colon\setR^{d\times d}\!\to \!\setR^{d\times d}_{\textup{sym}}}$, 
defined, for every $\BA\in \mathbb{R}^{d\times d}$, via
\begin{align}
\begin{aligned}
    \BF(\BA)&\coloneqq (\delta+\vert
    \BA^{\textup{sym}}\vert)^{\smash{\frac{p-2}{2}}}\BA^{\textup{sym}}\,. 
    \end{aligned}\label{eq:def_F}
\end{align}
The connection between
$\SSS,\BF\hspace{-0.05em}\colon\hspace{-0.05em}\setR^{d \times d}
\hspace{-0.05em}\to\hspace{-0.05em} \setR^{d\times d}_{\textup{sym}}$ and
$\phi_a,(\phi_a)^*\hspace{-0.05em}\colon\hspace{-0.05em}\setR^{\ge
  0}\hspace{-0.05em}\to\hspace{-0.05em} \setR^{\ge
  0}$,~${a\hspace{-0.05em}\ge\hspace{-0.05em} 0}$, is best explained
by the following result (cf.~\cite{die-ett,dr-nafsa,dkrt-ldg}).

\begin{prop}
  \label{lem:hammer}
  Let $\SSS$ satisfy Ass.~\ref{assum:extra_stress}, let
  $\varphi$ be defined in \eqref{eq:5}, and let $\BF$ be
  defined in \eqref{eq:def_F}. Then, uniformly with respect to  
  $\BA, \BB \in \setR^{d \times d}$, we have that\vspace{-1mm}
    \begin{align}\label{eq:hammera}
        \begin{aligned}
        \big(\SSS(\BA) - \SSS(\BB)\big)
      :(\BA-\BB ) &\sim  \abs{ \BF(\BA) - \BF(\BB)}^2
      \\
      &\sim \phi_{\abs{\BA^{\textup{sym}}}}(\abs{\BA^{\textup{sym}}
        - \BB^{\textup{sym}}})
      \,,
      \end{aligned}
    \end{align}
        \vspace{-4.5mm}
    \begin{align}
      \label{eq:hammere}
      \hspace{16mm}	\abs{\SSS(\BA) - \SSS(\BB)}
      &\sim   \smash{\varphi'_{\vert
        \BA^{\textup{sym}}\vert}(\vert\BA^{\textup{sym}}-\BB^{\textup{sym}}\vert
        )} 
    \end{align}
    with constants depending only on the characteristics of ${\SSS}$.
\end{prop} 
\begin{rem}\label{rem:sa}
  {\rm For the operators
    $\SSS_a\colon\mathbb{R}^{d\times
      d}\to\smash{\mathbb{R}_{\textup{sym}}^{d\times d}}$, $a \ge 0$,
    defined in \eqref{eq:flux}, the assertions of 
    Lem.~\ref{lem:hammer} hold with
    $\phi$ 
    replaced by
    $\phi_a $. 
  }
\end{rem}

The following results can be found in~\cite{DK08,dr-nafsa}.

\begin{lem}[Change of Shift]\label{lem:shift-change}
    Let $\varphi$ be defined in \eqref{eq:5} and let $\BF$ be defined in \eqref{eq:def_F}. Then,
  for each $\varepsilon>0$, there exists $c_\varepsilon\geq 1$ (depending only
  on~$\varepsilon>0$ and $p$) such that for every $\BA,\BB\in\smash{\setR^{d \times d}_{\textup{sym}}}$ and $t\geq 0$, it holds
  \begin{align*}
    \smash{\phi_{\abs{\BB}}(t)}&\leq \smash{c_\varepsilon\, \phi_{\abs{\BA}}(t)
    +\varepsilon\, \abs{\BF(\BB) - \BF(\BA)}^2\,,}
    \\
        \smash{\phi_{\abs{\BB}}(t)}&\leq \smash{c_\varepsilon\, \phi_{\abs{\BA}} (t)
    +\varepsilon\, \phi_{\abs{\BA}}\big(\bigabs{\abs{\BB} - \abs{\BA}}\big)\,,}
    \\
    \smash{(\phi_{\abs{\BB}})^*(t)}&\leq \smash{c_\varepsilon\, (\phi_{\abs{\BA}})^*(t)
                                      +\varepsilon\, \abs{\BF(\BB) - \BF(\BA)}^2}\,,
   \\
    \smash{(\phi_{\abs{\BB}})^*(t)}&\leq \smash{c_\varepsilon\, (\phi_{\abs{\BA}})^*(t)
    +\varepsilon\, \phi_{\abs{\BA}}\big(\bigabs{\abs{\BB} - \abs{\BA}}\big)}\,.
  \end{align*}
%
\end{lem}

\subsection{The $p$-Stokes system} 
Let us briefly recall some well-known~facts about the $p$-Stokes system \eqref{eq:p-stokes}. For $p\in (1,\infty)$, we define the function spaces
\begin{align*}
  \begin{aligned}
    V&\coloneqq (W^{1,p}(\Omega))^d\,,&&\Vo\coloneqq (W^{1,p}_0(\Omega))^d\,,\\
    Q&\coloneqq L^{p'}(\Omega)\,,&&\Qo\coloneqq
    L_0^{p'}(\Omega)\coloneqq \big\{f\in
    L^{p'}(\Omega)\;|\;\mean{f}_\Omega=0\big\}\,.
  \end{aligned}
\end{align*}
With this notation, the weak formulation of problem \eqref{eq:p-stokes} is the following:
    
\textbf{Problem (Q).} For given $\bff \hspace{-0.1em}\in \hspace{-0.1em}\Vo^*$, find
$(\bv,q) \hspace{-0.1em}\in \hspace{-0.1em}\Vo\times \Qo$ such that for all
${(\bz,z)\hspace{-0.1em}\in \hspace{-0.1em}\Vo\times Q}$, it holds
\begin{align}\label{eq:q1}
  (\SSS(\BD\bv),\BD\bz)_\Omega-(q,\divo\bz)_\Omega&=\langle \bff,\bz\rangle _{\Vo}\,,
  \\
  (\divo\bv,z)_\Omega&=0\label{eq:q2}\,.
\end{align}

Alternatively, we can reformulate Problem (Q) ``hiding'' the
pressure. To this end we define the spaces
\begin{align*}
    V_\divo\coloneqq \{\bz\in V\mid \divo \bz=0\}\,,\quad \Vo_\divo\coloneqq V_\divo\cap \Vo\,.
\end{align*}

\textbf{Problem (P).} For given $\bff \in \Vo^*$, find $\bv\in \Vo_\divo$ 
such~that~for~all~${\bz\in \Vo_\divo}$, it holds
\begin{align}\label{eq:p}
    (\SSS(\BD\bv),\BD\bz)_\Omega&=\langle \bff,\bz\rangle _{\Vo}\,.
\end{align}

Note that $\Vo_\divo\hspace{-0.17em}\neq\hspace{-0.17em} \emptyset$
by the solvability of the divergence equation
(cf.~\cite[Thm.~6.6]{john}). This and the theory of monotone operators
yield the existence of a weak solution of Problem~(P) if $p>
1$, and the apriori estimate
\begin{align}
  \label{eq:apri-cont}
  \|\bv \|_{1,p,\Omega} \lesssim \|\bff\|_{-1,p'} +\delta\, \abs{\Omega}\,.
\end{align}
De Rham's lemma 
then ensures the solvability of Problem (Q).
%

\subsection{DG framework}\label{sec:dg-space}

\subsubsection{Triangulations}

\!In \hspace{-0.1mm}what \hspace{-0.1mm}follows, \hspace{-0.1mm}we
\hspace{-0.1mm}always \hspace{-0.1mm}denote \hspace{-0.1mm}by
\hspace{-0.1mm}$\mathcal{T}_h$,
\hspace{-0.1mm}$h\!>\!0$,~\hspace{-0.1mm}a~\hspace{-0.1mm}\mbox{family}~\hspace{-0.1mm}of
uniformly \hspace{-0.1mm}shape \hspace{-0.1mm}regular
\hspace{-0.1mm}and \hspace{-0.1mm}conforming
\hspace{-0.1mm}triangulations
\hspace{-0.1mm}of~\hspace{-0.1mm}${\Omega\hspace{-0.15em}\subseteq\hspace{-0.15em}
  \mathbb{R}^d}$,~${d\hspace{-0.15em}\in\hspace{-0.15em}
  \set{2,3}}$,~\hspace{-0.1mm}cf.~\hspace{-0.1mm}\cite{BS08},~\hspace{-0.1mm}each
consisting of \mbox{$d$-dimensional} simplices $K$.
The collection of all mesh faces will be denoted $\Gamma_h$.
The parameter
${h>0}$, refers to the maximal mesh-size of $\mathcal{T}_h$, i.e., if we
define $h_K\coloneqq \textup{diam}(K)$ for all~${K\in
\mathcal{T}_h}$,~then ${h\coloneqq \max_{K\in
    \mathcal{T}_h}{h_K}}$. The (local) mesh-size function
$h_{\mathcal{T}}\colon \overline{\Omega}\to \mathbb{R}$ is defined via
${h_{\mathcal{T}}|_K\coloneqq h_K} $ for all $K\in \mathcal{T}_h$.
The (local) face-size function
$h_{\Gamma}\colon \Gamma_h\to \mathbb{R}$ is defined via
$h_{\Gamma}|_F\coloneqq h_F \coloneqq \text{diam}(F)$ for all
$F\in \Gamma_h$.  For simplicity, we always assume that $h \le 1$.
For a simplex $K\! \in\! \mathcal{T}_h$, we denote by $\rho_K\!>\!0$,
the supremum of diameters~of~inscribed~balls. We assume that
there~is~a~constant~$\omega_0\!>\!0$, independent of $h>0$, such that
${h_K}{\rho_K^{-1}}\le \omega_0$ for every $K \in \mathcal{T}_h$. The
smallest such constant~is~called~the~{\it chunkiness}~of
$(\mathcal{T}_h)_{h>0}$.  By $\Gamma_h^{i}$, we denote the~interior
faces and put $\Gamma_h\coloneqq \Gamma_h^{i}\cup \partial\Omega$. For
a face $F\vcentcolon= \partial K \cap \partial K'$, we use the
notation $S_F\vcentcolon= K \cup K'$.  
%
We introduce the following scalar
products~on~$\Gamma_h$
\begin{align*}
  \skp{f}{g}_{\Gamma_h} \coloneqq  \smash{\sum_{F  \in \Gamma_h} {\langle f, g\rangle_F }}\,,\quad\text{ where }\quad\langle f, g\rangle_F \coloneqq \int_F  f g \,\textup{d}s\quad\text{ for all }F \in \Gamma_h\,,
\end{align*}
if all the integrals are well-defined. Similarly, we define the products 
$\skp{\cdot}{\cdot}_{\partial\Omega}$ and~$\skp{\cdot}{\cdot}_{\Gamma_h^{i}}$. 

\subsubsection{Broken function spaces and projections}
For every $\ell \hspace{-0.1em}\in\hspace{-0.1em}
\setN_0$~and~${K\hspace{-0.1em}\in\hspace{-0.1em} \mathcal{T}_h}$, we
denote by ${\mathcal P}_\ell(K)$, the space of polynomials of degree at
most $\ell$ on $K$. Moreover, we set $\mathcal P_{-1}(K)\coloneqq
\{0\}$ for any ${K\hspace{-0.1em}\in\hspace{-0.1em} \mathcal{T}_h}$. Then, for given~$p\in (1,\infty)$, and
$k \in \setN_0$, we define the spaces
\begin{align*}
    Q_h^k&\coloneqq \big\{ q_h\in L^1(\Omega)\fdg q_h|_K\in \mathcal{P}_k(K)\text{ for all }K\in \mathcal{T}_h\big\}\,,\\
    V_h^k&\coloneqq \big\{\bv_h\in L^1(\Omega)^d\fdg \bv_h|_K\in \mathcal{P}_k(K)^d\text{ for all }K\in \mathcal{T}_h\big\}\,,\\
    X_h^k&\coloneqq \big\{\BX_h\in L^1(\Omega) ^{d \times d}\fdg  \BX_h|_K\in \mathcal{P}_k(K) ^{d \times d}\text{ for all }K\in \mathcal{T}_h\big\}\,,\\
        W^{1,p}(\mathcal T_h)&\coloneqq \big\{\bw_h\in
        L^1(\Omega)^d\fdg \bw_h|_K\in W^{1,p}(K)^d\text{ for all
        }K\in \mathcal{T}_h\big\}\,.
\end{align*}
In addition, we define the spaces $W^{1,\psi}(\mathcal T_h)\!\coloneqq \!\big\{\bw_h\!\in \!L^1(\Omega)^d\mid
\bw_h|_K\!\in\! W^{1,\psi}(K)^d\text{ for all }K\in
\mathcal{T}_h\big\}$, $X_h^{\smash{k,\textup{sym}}}\coloneqq X_h^k\cap
  L^1(\Omega;\mathbb{R}^{d\times d}_{\textup{sym}})$,  and
$\Qhko\coloneqq Q_h^k\cap \Qo$. 

For every $\bw_h\hspace{-0.1em}\in\hspace{-0.1em} \WDG$, we denote by
$\nabla_h \bw_h\hspace{-0.1em}\in\hspace{-0.1em} L^p(\Omega)^{d\times
  d}$, the \textit{local gradient},~defined via
$(\nabla_h \bw_h)|_K\hspace{-0.1em}\coloneqq
\hspace{-0.1em}\nabla(\bw_h|_K)$
for~all~${K\hspace{-0.1em}\in\hspace{-0.1em}\mathcal{T}_h}$.  For
every $K\hspace{-0.1em}\in\hspace{-0.1em} \mathcal{T}_h$,
${\bw_h\hspace{-0.1em}\in\hspace{-0.1em} \WDG}$~admits~an
interior trace ${\textrm{tr}^K(\bw_h)\in L^p(\partial K)^d}$. For each
face $F \in \Gamma_h$ of a given simplex $K\in \mathcal{T}_h$, we
define this interior trace by
$\smash{\textup{tr}^K_F (\bw_h)\in L^p(F )^d}$. Then, for every
$\bw_h\in \WDG$ and interior faces $F \in \Gamma_h^{i}$ shared by
adjacent elements $K^-_F , K^+_F \in \mathcal{T}_h$, we~define the \textit{average} and \textit{normal jump}, resp., of $\bw_h$ on
$F $, via 
\begin{align*}
  \{\bw_h\}_F &\coloneqq \smash{\frac{1}{2}}\big(\textup{tr}_F ^{K^+}(\bw_h)+
  \textup{tr}_F ^{K^-}(\bw_h)\big)\in  L^p(F )^d\,,
  \\
  \llbracket\bw_h\otimes\bn\rrbracket_F 
  &\coloneqq \textup{tr}_F ^{K^+}(\bw_h)\otimes\bn^+_F +
    \textup{tr}_F ^{K^-}(\bw_h)\otimes\bn_F ^-  \in L^p(F )^{d\times d}\,,
\end{align*}
where $\bn_F ^\pm $ denote the outward unit normal vector of $K_F^\pm$.  Moreover, for every $\bw_h\in \WDG$ and boundary faces
$F \!\in\! \partial\Omega$, we define {\it boundary averages} and
{\it boundary~jumps},~resp.,~via
\begin{align*}
  \{\bw_h\}_F &\coloneqq \textup{tr}^\Omega_F (\bw_h) \in L^p(F)^d\,,
  \\
  \llbracket \bw_h\otimes\bn\rrbracket_F
   &\coloneqq  \textup{tr}^\Omega_F (\bw_h)\otimes\bn \in L^p(F )^{d\times d}\,,
\end{align*}
where $\bn\colon\partial\Omega\to \mathbb{S}^{d-1}$ denotes the unit
normal vector field to $\Omega$ pointing outward.  Analogously, we
define $\{\BX_h\}_F $ and $ \llbracket\BX_h\bn\rrbracket_F $~for
all $\BX_h \in \Xhk$ and $F \in \Gamma_h$. Furthermore, if there is
no danger~of~confusion, we will omit the index
$F \in \Gamma_h$,~in~particular, if we interpret jumps and averages as
global functions defined on the whole of $\Gamma_h$.
In addition, for every $\bw_h\in \WDG$, we  introduce the \textit{DG norm}, via
\begin{align*}
    \|\bw_h\|_{\nabla,p,h}\coloneqq \|\nabla_h\bw_h\|_{p,\Omega}+
    \big\|h_\Gamma^{-1/p'}\jump{\bw_h\otimes \bn}\big\|_{p,\Gamma_h}\,,
\end{align*}
which turns $\WDG$ into a Banach space\footnote{The completeness of
  $\WDG$ equipped with $\|\cdot\|_{\nabla,p,h}$, for every fixed
  $h>0$, follows from ${\|\bw_h\|_p\leq c\,\|\bw_h\|_{\nabla,p,h}}$
  for all $\bw_h\in \smash{\WDG}$ (cf.~\cite[Lem.~A.9]{dkrt-ldg})
  and an element-wise application of the trace
  theorem.}.

\subsubsection{DG gradient and jump operators}

For every $k\in \mathbb{N}_0$, we define the \textit{(global) jump
  operator} $\smash{\Rhk \colon\WDG \to X_h^k}$ (using Riesz
representation) for every $\bw_h\in \smash{\WDG}$ via 
\begin{align*}
  \big(\Rhk\bw_h,\BX_h\big)_\Omega=\big\langle
  \llbracket\bw_h\otimes\bn\rrbracket,\{\BX_h\}\big\rangle_{\Gamma
  _h}\quad \textrm{for all }\BX_h\in X_h^k\,.
\end{align*}
In addition, for every $k\in \mathbb{N}_0$, the \textit{DG gradient operator} 
$  \Ghk\colon\WDG\to L^p(\Omega)$ is  defined, for every $\bw_h\in \smash{\WDG}$, via
\begin{align}
  \boldsymbol{\mathcal G}^k_{h}\bw_h\coloneqq 
  \nabla_h\bw_h-\boldsymbol{\mathcal R}^k_h\bw_h
  \quad\text{ in }L^p(\Omega)^{d\times d}\,.\label{eq:DGnablaR} 
\end{align}
Note that for every $\bv\in \Vo$, we have that $\Ghk \bv=\nabla\bv
$ in $L^p(\Omega)^{d\times d}$.
Owing to \cite[{(A.26)--(A.28)}]{dkrt-ldg},
for every $p\in (1,\infty)$ and $k\in \setN$, there exists a constant
$c>0$ such that for every $\bw_h\in \smash{\WDG}$, it holds
\begin{align}\label{eq:eqiv0}
    c^{-1}\,\|\bw_h\|_{\nabla,p,h}\leq
  \big\|\Ghke\bw_h\big\|_{p,\Omega}+
\big\|h_\Gamma^{-1/p'}\jump{\bw_h\otimes
  \bn}\big\|_{p,\Gamma_h}\leq c\,\|\bw_h\|_{\nabla,p,h}\,. 
\end{align}
For a generalized N-function $\psi$, we define the
pseudo-modular\footnote{The definition of a pseudo-modular can be
  found in \cite{Mu}.}
$\smash{m_{\psi,h}}\colon W^{1,\psi}(\mathcal T_h)\to \mathbb{R}^{\ge
  0}$, for every $\bw_h\in W^{1,\psi}(\mathcal T_h)$ via
\begin{align}\label{eq:mod1}
  m_{\psi,h}(\bw_h) &\coloneqq
 \int_{\Gamma_h}h_\Gamma \psi(h_\Gamma^{-1}\jump{\bw_h\otimes \bn})
\end{align}
For $\psi = \phi_{p,0}$, we have that
$m_{\psi,h}(\bw_h)=\|h_\Gamma^{-1/p'}\jump{\bw_h\otimes
  \bn}\|_{p,\Gamma_h}^p$~for~all~${\bw_h\in W^{1,\psi}(\mathcal T_h)}$.

\subsubsection{Symmetric \hspace{-0.1mm}DG \hspace{-0.1mm}gradient \hspace{-0.1mm}and \hspace{-0.1mm}symmetric \hspace{-0.1mm}jump \hspace{-0.1mm}operators}

For ${\bw_h\hspace{-0.1em}\in\hspace{-0.1em} \WDG}$, we denote by
$\BD_h\hspace{-0.1em}\bw_h\hspace{-0.1em}\coloneqq
\hspace{-0.1em}[\nabla_h\bw_h]^{\textup{sym}}\hspace{-0.1em}\in
\hspace{-0.1em}L^p(\Omega;\mathbb{R}^{d\times d}_{\textup{sym}})$, the
\textit{local symmetric gradient}, and define the \textit{symmetric DG
  gradient operator}
$ \smash{\Dhk\!\colon\!\WDG \!\to\! L^p(\Omega;\mathbb{R}^{d\times
    d}_{\textup{sym}})}$ 
via
$\smash{\Dhk\bw_h\hspace{-0.1em}\coloneqq
  \hspace{-0.1em}[\Ghk\bw_h]^{\textup{sym}}}
  $, i.e., if we define, for every ${k\in \setN_0}$, 
the \textit{symmetric jump operator}
$\Rhks\colon\WDG\to X_h^{\smash{k,\textup{sym}}}$ via
$\Rhks\bw_h\coloneqq [\Rhk\bw_h]^{\textup{sym}}\in
  X_h^{\smash{k,\textup{sym}}}$, for every ${\bw_h\in
  \WDG}$, we have that
\begin{align*}
  \smash{\Dhk\bw_h =\BD_h\bw_h
  -\Rhks\bw_h
  \quad\text{ in }L^p(\Omega;\mathbb{R}^{d\times d}_{\textup{sym}})\,.}
\end{align*}
Additionally, for every $\bw_h\in \WDG$, we introduce the
\textit{symmetric DG norm}~via
\begin{align*}
  \smash{\|\bw_h\|_{\BD,p,h}\coloneqq \|\BD_h\bw_h\|_{p,\Omega}
  +\smash{\big\|  h_\Gamma^{-1/p'} \llbracket\bw_h\otimes\bn\rrbracket\big\|_{p,\Gamma_h}}}\,.
\end{align*}
We make frequent use of the discrete Korn inequality on $V_h^k$
(cf.\ \cite[Prop.~2.4]{kr-pnse-ldg-1}). 

\begin{prop}[Discrete Korn inequality]\label{korn}
  For every $p\in (1,\infty)$ and $k\in \setN$, there holds for every
  $\bv_h\in V_h^k$
  \begin{align*}
    \smash{\|\bv_h\|_{\nabla,p,h}\lesssim\|\bv_h\|_{\BD,p,h}}
  \end{align*}
  with a constant depending only on $k$, $\omega_0$ and $p$. 
\end{prop}
For the symmetric DG norm holds a similar relation like
\eqref{eq:eqiv0} (cf.~\cite[Prop.~2.5]{kr-pnse-ldg-1}), namely for
every $p\in (1,\infty)$ and $k\in \setN$, there exists a
constant~${c>0}$ such that for every $\bw_h\in \WDG$, it holds
        \begin{align}
          \smash{c^{-1}\,\|\bw_h\|_{\BD,p,h}
          \leq \big\|\Dhke\bw_h\big\|_{p,\Omega}
+\big\|  h_\Gamma^{-1/p'}\llbracket\bw_h\otimes\bn\rrbracket\big\|_{p,\Gamma_h}
          \leq c\,\|\bw_h\|_{\BD,p,h}\,.}
          \label{eq:equi2}
        \end{align}

\subsubsection{DG divergence operator}
The \textit{local divergence} is defined, for every
$\bw_h\!\in\!\WDG$, via
$\textup{div}_h\bw_h\coloneqq \text{tr}(\nabla_h \bw_h)\!\in\!
L^p(\Omega)$.  In addition, for every~${k\in \setN_0}$, the
\textit{DG divergence operator} \linebreak
$\Divhk\hspace{-0.17em}\colon\hspace{-0.17em}\WDG\hspace{-0.17em}\to
\hspace{-0.17em}L^p(\Omega)$ \hspace{-0.2mm}is
\hspace{-0.2mm}defined,~\hspace{-0.2mm}for~\hspace{-0.2mm}every~\hspace{-0.2mm}${\bw_h\hspace{-0.17em}\in\hspace{-0.17em}
  \WDG}$, via
$\Divhk\bw_h\coloneqq \text{tr}(\Ghk\bw_h)
  =\text{tr}(\Dhk\bw_h)$ $\in L^p(\Omega)$, i.e.,
\begin{align*}    
  \smash{\Divhk\bw_h=\textup{div}_h\bw_h-\textup{tr}(\Rhk\bw_h)\quad\text{ in }L^p(\Omega)\,.}
\end{align*}
In particular, for every $\bw_h\in \WDG$ and $z_h \in \smash{Q_h^k}$,
we have that
\begin{align*}
    \big(\Divhk\bw_h,z_h\big)_\Omega&=(\textup{div}_h\bw_h,z_h)_\Omega
    -\big\langle \llbracket
    \bw_h\cdot\bn\rrbracket,\{z_h\}\big\rangle_{\Gamma_h}
    \\[-0.5mm]
    &=-(\bw_h,\nabla_h z_h)_\Omega +\big\langle \{\bw_h\cdot\bn\},
    \llbracket z_h\rrbracket\big\rangle_{\Gamma_h^{i}}    \,.
\end{align*}
and thus, for every $\bv_h\in \Vhk$ and $z_h \in \Qhke $
it holds
\begin{align}
  \label{eq:div-dg.0a}
  \begin{aligned}
    \big(\Divhke\bv_h,z_h\big)_\Omega&=- \big (\bv_h, \mathcal {G}_h^k z_h\big )_\Omega\,.
  \end{aligned}
\end{align}
\subsubsection{The smoothing operator}\label{sec:smoothing}

We employ here a smoothing operator
$\BE_h \coloneqq \BE_{h}^k \colon \Vhk \to \Vo$, $k \in
\setN$, 
similar to the one introduced in \cite{KZ.2020,KVZ.2021}.  The operator, for every
$\bw_h\in \Vhk$, is constructed as:
\begin{alignat*}{2}
  \BE_h \bw_h &\coloneqq \BE_h^{(1)} \bw_h + \BE_h^{(2)} \bw_h
 +\BE_h^{(3)} \bw_h, 
\end{alignat*}
where the aim is that the following conditions are satisfied:
\begin{itemize}
  \item $\BE^{(1)}_h$ ensures that $\BE_h$ is stable in an Orlicz--Sobolev sense;
    see Prop.~\ref{prop:n-function_E} below.
  \item $\BE^{(2)}_h$ ensures that $\BE_h$ preserves $(k-1)$-moments on facets:
    \begin{equation}\label{eq:facet_moments}
      (\BE_h\bw_h, \bz_F)_F 
      = 
      (\leftavg\bw_h\rightavg,\bz_F)_F
      \qquad \forall\, F \in \facetsint,\, \bz_F \in \mathcal{P}_{k-1}(F)^d.
    \end{equation}
    
  \item $\BE_h^{(3)}$ ensures that $\BE_h$ preserves the discrete divergence:
        \begin{align}
          \label{eq:divo}
          \Divhke \bz_h =\divo \BE_h\bz_h
          \qquad \forall\, \bz_h\in \Vhk.
        \end{align}
\end{itemize}

where each operator is defined as:
	
          %
	\begin{itemize}
		\item \textit{$\BE_h^{(1)}$ (simplified nodal averaging):} Denote by $\mathcal{L}^{\mathrm{int}}_k(\mathcal{T}_h)$ the set of interior Lagrange nodes on $\mathcal{T}_h$ of degree $k$ and for every $y\in  \mathcal{L}^{\mathrm{int}}_k(\mathcal{T}_h)$ by $\varphi_y^k\in \mathcal{P}_k(\mathcal{T}_h)$,~the~\mbox{corresponding} basis function. In~addition, for every 
      $y \hspace*{-0.1em} \in\hspace*{-0.1em} \mathcal{L}^{\mathrm{int}}_k(\mathcal{T}_h)$  fix an arbitrary  $K_y \hspace*{-0.1em} \in\hspace*{-0.1em}  \mathcal{T}_h$~such~that~$y \hspace*{-0.1em} \in\hspace*{-0.1em}  K_y $. Then, the operator 
      $\BE^{(1)}_h\colon V_h^k \to V_h^k\cap \Vo$,
      for every $\bw_h\in \Vhk$ is defined as:
		\begin{align*}
      \BE^{(1)}_h  (\bw_h) \coloneqq  \sum_{y \in \mathcal{L}^{\mathrm{int}}_k(\mathcal{T}_h)} (\bw_h|_{K_y })(y ) \varphi_y ^k
		\end{align*}
		\item \textit{$\BE_h^{(2)}$ (facet bubble smoother):} Denote for
		$F\in \Gamma_h$ by $\mathcal{L}_{k-1}(F)$ the set of Lagrange nodes on $F$~of~degree $k-1$ and by $\varphi_F \coloneqq \varphi_{y_1}^1\cdot \ldots\cdot\varphi_{y_d}^1 \in \mathcal{P}_d(F)\cap W^{1,1}_0(F)$, where $y_1,\dots,y_d\in \mathcal{L}_1(F)$ are such that
		$F=\textup{conv}\{y_1,\dots, y_d\}$, the corresponding facet bubble function.
		Then, the operator $\BE_h^{(2)} \colon 
\Vhk \to \mathcal P_{k+d-1}(\tria)^d \cap \Vo$,
for every $\bw_h\in \Vhk$,~is~\mbox{defined}~via
		\begin{align*}
			\BE^{(2)}_h(\bw_h) \coloneqq  
            \smashoperator[l]{\sum_{F \in \Gamma_h^{i}}}
            \smashoperator[r]{\sum_{y \in \mathcal{L}_{k-1}(F)}}
           \boldsymbol{\mathcal{Q}}_F( \leftavg \bw_h \rightavg - \BE_h^{(1)}(\bw_h))(y  )\, \varphi^{k-1}_y \, \varphi_F
		\end{align*}
		where $\boldsymbol{\mathcal{Q}}_F\!\colon \!L^1(F)^d \!\to \! \mathcal{P}_{k-1}(F)^d$ is the weighted
		$L^2$-projection,
    defined for every ${\bw_h\!\in\! L^1(F)^d}$ by requiring:
		\begin{align*}
			(\varphi_F\boldsymbol{\mathcal{Q}}_F \bw_h, \bz_F)_F =( \bw_h, \bz_F)_F
      \qquad
      \forall\, \bz_F\in \mathcal{P}_{k-1}(F)^d.
		\end{align*}
		\item 
                  \textit{$\BE_h^{(3)}$ (divergence correction):}
                  For any $K\in \triang$, denote by $\triaMK$ the collection of elements that result from the barycentric refinement of $K$; define the local `velocity' and `pressure' spaces for $\ell\geq d$:
                  \begin{equation*}
                    \Vo^{\ell}_h(\triaMK) 
                    \coloneqq \mathcal{P}_\ell(\triaMK) \cap W^{1,1}_0(K)
                    \qquad 
                    \Qo^{\ell-1}_h(\triaMK) 
                    \coloneqq 
                    \mathcal{P}_{\ell -1}(\triaMK)\cap L^1_0(K).
                  \end{equation*}
That is, $\Vo^\ell_h(\triaMK)\times \Qo^{\ell-1}_h(\triaMK)$ consists of a Scott--Vogelius pair with homogeneous boundary data on the Alfeld split of $K$; it is well-known that this pair is inf-sup stable for $\ell\geq d$.
With this we define the local right-inverse of the divergence $\BE^{(3)}_K \colon \Qo_h^{k+d-2}(\triaMK)\to \Vo_h^{k+d-1}(\triaMK)$ through:
\begin{equation*}
  \BE_K^{(3)}(q_K) = 
  \argmin \{ \norm{\nabla \bw_K}_{2,K}^2 \mid \bw_K \in \Vo_h^{k+d-1}(\triaMK) \, \text{ with }\diver\bw_K =q_K\}.
\end{equation*}
Abusing the notation, we can consider $\BE^{(3)}_K(q_K)$ to be defined on $\Omega$ by extending by zero.
Thus, the divergence correction is defined for any $\bw_h\in V_h^k$ as:
\begin{equation*}
  \BE_h^{(3)}(\bw_h)
  \coloneqq 
\sum_{K\in \tria} \BE^{(3)}_K(\Divhke \bw_h - \diver (\BE_h^{(1)}\bw_h + \BE^{(2)}_h\bw_h))
\end{equation*}

	\end{itemize}

  We note that the smoothing operator from \cite{KZ.2020,KVZ.2021} includes for $k\geq 3$ an additional component $\BE_h^{(4)}$ that ensures that $\BE_h$ preserves $(k-2)$-moments in the interior of each element.
  One consequence of this is that the smoothing operator is only required to appear in the forcing term (see Rem.~\ref{rem:grad_tilde} below);
  since for degree $k\geq 1$, nonlinear functions of piecewise-polynomials are not necessarily piecewise-polynomial anymore,
  in this work this interior moment preservation property does not bring any advantages.


In \cite[Prop.~18]{VZ.2018.III} the following local
$L^2$-approximation property of $\BE_h=\BE_h^k$ is proved: There exists a constant
depending only on $k$ and $\omega_0$ such that for every
$\bv_h\in \Vhk$ and  $K \in \mathcal{T}_h$ we have
\begin{align}\label{eq:E_stability_L2}
  \|\bv_h - \BE_h  \bv_h \|_{2,K}
  \lesssim  \sum_{F\in \Gamma_h(K)}{\|h_F^{-1/2} \jump{\bv_h \otimes \bn}_F\|_{2,F}}\,,
\end{align}
where
$\Gamma_h(K)\coloneqq \{F\in \Gamma_h\mid K\cap
F\neq\emptyset\}$. Here we also used that
$\abs{\jump{\bv_h}}=\abs{\jump{\bv_h \otimes \bn}}$.
We now proceed to generalise the estimate \eqref{eq:E_stability_L2} to
the Orlicz setting.
See also \cite[Lem.~2.12]{DHKZ.2025} for a similar result with Crouzeix--Raviart elements.
	
\begin{prop}[Orlicz approximability and stability]\label{prop:n-function_E}
  Let $\psi\colon \mathbb{R}^{\ge 0}\to\mathbb{R}^{\ge 0} $ be an
  N-funct\-ion with $\psi\in \Delta_2$, $\psi^*\in \Delta_2$, and let
  $k\in \mathbb{N}$. Then, for every $\bv_h\in \Vhk$ and
  $K\in \mathcal{T}_h$, it holds that
  \begin{align}\label{eq:stab-orlicz}
    &\dashint_K{\psi(\vert \Ghke (\bv_h-\BE_h
    \bv_h)\vert)\,\textup{d}x}\lesssim  \sum_{F\in
    \Gamma_h(K)}\dashint_F{\psi(\vert h_F^{-1}\jump{\bv_h
    \otimes \bn}_F\vert)\, \textup{d}s}\,,
    \\
        &\dashint_K{\psi(\vert \nabla \BE_h
    \bv_h\vert)\,\textup{d}x}\lesssim  \dashint_K{\psi(\vert \nabla _h
    \bv_h\vert)\,\textup{d}x} +\sum_{F\in
    \Gamma_h(K)}\dashint_F{\psi(\vert h_F^{-1}\jump{\bv_h
    \otimes \bn}_F\vert)\, \textup{d}s}\,, \label{eq:stab-orlicz1}
  \end{align}
  where the constants depend only on $k$,
  $\Delta_2(\psi),\Delta_2(\psi^*)$ and $\omega_0$.
\end{prop}

\begin{proof}
  First, we prove \eqref{eq:stab-orlicz} with $\Ghke$ replaced by the
  local gradient $\nabla_h$.
  Owing to \eqref{eq:E_stability_L2} together with \cite[Lem.\
  12.1]{EG21} and $h_F\sim h_K$, we have that
  \begin{align}
    \| \nabla_h (\bv_h-\BE_h  \bv_h)\|_{\infty,K}\lesssim \sum_{F\in
    \Gamma_h(K)}\dashint_F{\vert h_F^{-1}\jump{\bv_h
    \otimes \bn}_F\vert\,\textup{d}s}\,,\label{prop:n-function_E.1} 
  \end{align}
  with a constant depending only on $k$ and $\omega_0$.  Using
  Prop.~\ref{prop:n-function_E.1}, the $\Delta_2$-condition and convexity
  of $\psi$, in particular, Jensen's inequality, and that
  ${\sup_{h>0}\sup_{K\in
    \mathcal{T}_h}{\textup{card}(\Gamma_h(K))}\allowbreak\leq c}$,
  where $c>0$ depends only on $\omega_0$, we find that
  \begin{align}\label{prop:n-function_E.2}
    \dashint_K{\psi(\vert \nabla_h (\bv_h-\BE_h  \bv_h)\vert)\,\textup{d}x}
    &\lesssim \psi\bigg(\frac{1}{\textup{card}(\Gamma_h(K))} \sum_{F\in\Gamma_h(K)}\dashint_F{\vert
      h_F^{-1}\jump{\bv_h \otimes \bn}_F\vert \,\textup{d}s}\bigg )\notag
    \\
    &\lesssim
      \frac{1}{\textup{card}(\Gamma_h(K))}\sum_{F\in\Gamma_h(K)}\dashint_F{\psi(\vert
      h_F^{-1}\jump{\bv_h \otimes \bn}_F\vert)\,\textup{d}s} 
    \\
    &\lesssim \sum_{F\in\Gamma_h(K)}\dashint_F{\psi(\vert h_F^{-1}\jump{\bv_h \otimes
      \bn}_F\vert)\,\textup{d}s}\,,\notag 
  \end{align}
  where the constants depend only on $k$, $\Delta_2(\psi),\Delta_2(\psi^*)$ and $\omega_0$.
  This implies immediately \eqref{eq:stab-orlicz1} using $\nabla
  \BE_h  \bv_h=\nabla_h (\BE_h\bv_h- \bv_h) +\nabla _h \bv_h$ and the
  convexity of $\psi$.
    Appealing to \cite[(A.23)]{dkrt-ldg}, for every $F\in \Gamma_h$ and
  $K\in \mathcal{T}_h$ with $F\subseteq \partial K$, it holds that
  \begin{align}
    \dashint_K{\psi(\vert \boldsymbol{\mathcal{R}}_h^{k-1}
    \bv_h\vert)\,\textup{d}x}\lesssim \dashint_F{\psi(\vert h_F^{-1}\jump{\bv_h
    \otimes \bn}_F\vert)\,\textup{d}s}\,,\label{prop:n-function_E.3} 
  \end{align}
  where the constants depend only on $k$,
  $\Delta_2(\psi),\Delta_2(\psi^*)$ and $\omega_0$.  Thus, using
  \eqref{eq:DGnablaR}, Prop.~\ref{prop:n-function_E.2}, that
  ${\boldsymbol{\mathcal{R}}_h^{k-1} \BE_h \bv_h=\boldsymbol{0}}$, and
  Prop.~\ref{prop:n-function_E.3}, we obtain
  \begin{align*}
    \dashint_K{\psi(\vert \Ghke   (\bv_h-\BE_h  \bv_h)\vert)\,\textup{d}x}
    &\lesssim \dashint_K{\psi(\vert
      \nabla_h (\bv_h-\BE_h  \bv_h)\vert)\,\textup{d}x}
    +\dashint_K\!{\psi(\vert
      \boldsymbol{\mathcal{R}}_h^{k-1} \bv_h\vert)\,\textup{d}x}
    \\
    &\lesssim \sum_{F\in\Gamma_h(K)}\dashint_F{\psi(\vert h_F^{-1}
    \jump{\bv_h \otimes \bn}_F\vert)\,\textup{d}s}\,, 
  \end{align*}
  where the constants depend only on $k$, $\Delta_2(\psi),\Delta_2(\psi^*)$ and $\omega_0$.
\end{proof}

\section{LDG formulations}\label{sec:ldg}
To obtain a LDG formulation of \eqref{eq:p-stokes} we use ideas from
\cite{KVZ.2021,bgkr-optimal-dg,kr-pnse-ldg-1} to get 
the discrete counterpart~of~Problem~(Q):
    
    \textbf{Problem (Q$_h$).} For given $k \in \setN$ and $\bff \in \Vo^*$,
    find $(\bv_h,q_h)\!\in V_h^k\times \Qhkeo$ such that 
     for all $(\bz_h,z_h) \in \Vhk\times\Qhkeo$, it holds
    \begin{align}
      &\begin{aligned}
        &\bighskp{\SSS(\Dhke \bv_h )}{\BD \BE_h^k \bz_h}_\Omega
        +\alpha \bigskp{\SSS_{\bbeta_h(\bv_h)}(h_\Gamma^{-1}
          \jump{\bv_h \otimes \bn})}{ \jump{\bz_h \otimes
            \bn}}_{\Gamma_h}
        \\
        &\quad -\bighskp{q_h}{\Divhke \bz_h}_\Omega =  \langle \bff, \BE_h^k \bz_h\rangle _{\Vo} \,,
        \label{eq:primal0.1} 
      \end{aligned}
          \\
      \label{eq:primal0.2}&\bighskp{\Divhke \bv_h}{z_h}_\Omega=0\,,
    \end{align}
    where $\alpha\hspace*{-0.15em}>\hspace*{-0.15em}0$ is the
    \textit{stabilisation parameter},
    $\bbeta_h\colon
    \hspace*{-0.15em}V_h^k\hspace*{-0.15em}\to\hspace*{-0.15em}
    \mathbb{R}^{\ge 0}$ is the \textit{max-shift
      functional},~for~\mbox{every}~${\bv_h\hspace*{-0.15em}\in\hspace*{-0.15em}
      V_h^k}$, defined via
    \begin{align}\label{def_shift}
      \bbeta_h(\bv_h) \coloneqq  \begin{cases}
        0 & \text{if }p\leq 2\,,\\
        \|\Dhke \bv_h\|_{\infty,\Omega}& \text{if }p>2\,,
      \end{cases}
    \end{align}
    and $\SSS_{\bbeta_h(\bv_h)}$ is defined in \eqref{eq:flux}.  Next, we
    eliminate in the system~\eqref{eq:primal0.1},
    \eqref{eq:primal0.2}, the variable $q_h\in \Qhkeo$ to derive a
    system only expressed in terms of the discrete velocity
    $\bv_h\hspace{-0.1em}\in\hspace{-0.1em}
    \smash{\Vhk}$. To~this~end we introduce
    \begin{align*}
      V_h^k(0)&\coloneqq \big\{\bv_h\in V_h^k\mid \bighskp{\Divhke
                \bv_h}{z_h}_\Omega=0\textup{ for all }z_h\in
                \Qhke\big\}
      \\
                &\,= \big\{\bv_h\in V_h^k\mid \Divhke \bv_h=0 \big\}\,.
    \end{align*}
    Note that $V_h^k(0)\neq \emptyset$. In fact, in \cite{KVZ.2021} it
    is shown, using the inf-sup stability of the pair
    $(\Vhk, \Qhkeo)$, $k \in \setN$, (cf.~\cite{JLMNR17}), that the
    problem \eqref{eq:primal0.1}, \eqref{eq:primal0.2} has a solution in
    the case $p=2$. This solution belongs to $V_h^k(0)$. Using $V_h^k(0)$   we~obtain~the~discrete counterpart of Problem (P):
    
    \textbf{Problem (P$_h$).} For given $k \in \setN$ and 
    $\bff \in \Vo^*$, find $\bv_h\in V_h^k(0)$ such that for all
    $\bz_h\in     V_h^k(0)$, it holds
    \begin{align} \label{eq:primal}
        \begin{aligned}
          &\hspace*{-1.5mm}\bighskp{\SSS(\Dhke \bv_h )}{\BD \BE_h^k \bz_h}_\Omega
        \!+\!\alpha \bigskp{\SSS_{\bbeta_h(\bv_h)}(h_\Gamma^{-1}
          \jump{\bv_h \otimes \bn})}{ \jump{\bz_h \otimes
            \bn}}_{\Gamma_h}\!=  \langle \bff, \BE_h^k \bz_h\rangle _{\Vo} \,. \hspace*{-1.5mm}
        \end{aligned}
    \end{align}
    Problem (Q$_h$) and Problem (P$_h$) are called \textit {primal
      formulations} of the system~\eqref{eq:p-stokes}.
    
    \begin{rem}\label{rem:grad_tilde}
      {\rm
          (i) Note that since $\BE_h\bw_h \in \Vo$ we can interchangeably write $\Dhke\BE_h\bw_h=\BD_h\BE_h\bw_h=\BD\BE_h\bw_h$ and $\Ghke\BE_h\bw_h=\nabla_h\BE_h\bw_h=\nabla\BE_h\bw_h$ for all $\bw_h\in \Vhk$.

      (ii) Furthermore, from the moment preservation property \eqref{eq:facet_moments} and the definition of the DG gradient,
        one readily sees that the following holds for any $\BA_h \in X_h^{0,\mathrm{sym}}$ and $\bw_h\in V_h^1$:
        \begin{equation}
          (\BA_h, \BD \BE_h\bw_h )_\Omega 
          = (\BA_h, \nabla \BE_h\bw_h)_\Omega 
          = (\BA_h, \boldsymbol{\mathcal{G}}^0_h\bw_h)_\Omega
          = (\BA_h, \boldsymbol{\mathcal{D}}_h^0\bw_h)_\Omega,
        \end{equation}
        which means that in the discrete formulations \eqref{eq:primal0.1} and \eqref{eq:primal},
        assuming that $k=1$, we could swap the first term for 
        \begin{equation}
          (\BS(\boldsymbol{\mathcal{D}}^0_h \bv_h), \boldsymbol{\mathcal{D}}^0_h\bz_h)_\Omega.
        \end{equation}
        In other words, with piecewise-linear approximations,
        it is only necessary to implement the smoothing operator in the forcing term.
        The same observation clearly holds also for the modified scheme that will be introduced in Section \ref{sec:rob}.

(iii) The max-shift functional is constructed analogously to
        the one in \cite{bgkr-optimal-dg} and has analogous
        properties. 
      }
      \end{rem}

	
    More precisely, the max-shift functional is precisely constructed in such a way that the following lemma applies.
	
    \begin{lem}\label{lem:max-shift}
      For every $\kappa>0$, there exists a constant $c_\kappa>0$,
      depending on the characteristics of~$\SSS$, such that for every
      $\bv_h,\bz_h\in V_h^k$ and $\bv\in V$, it holds that
      \begin{align*}
        &\vert (\SSS(\Dhke \bv_h) - \SSS(\BD \bv), \Dhke(\BE_h  \bz_h - \bz_h))_{\Omega}\vert \\
        &\lesssim \kappa \,\|\BF(\Dhke \bv_h) - \BF(\BD \bv)\|^2_{2,\Omega
          }+
          c_\kappa\, m_{\varphi_{\bbeta_h(\bv_h)},h}(\bz_h)\,,
      \end{align*}
      where the constants depend only on $k$, $\omega_0$, and the characteristics of $\BS$.
    \end{lem}
	
    \begin{proof}
      Using \eqref{eq:hammere} and the $\varepsilon$-Young inequality
      \eqref{ineq:young} with $\psi=\varphi_{\bbeta_h(\bv_h)}$, we find
      that
      \begin{align}
        \label{eq:max-shift.1}
          &\vert (\SSS(\Dhke \bv_h) - \SSS(\BD \bv),
          \Dhke(\BE_h \bz_h - \bz_h))_{\Omega}\vert\\
          &\le          c_\kappa\,\rho_{\varphi_{\bbeta_h(\bv_h)},\Omega}(\Dhke(\BE_h
          \bz_h - \bz_h)) +\kappa\,
          \rho_{(\varphi_{\bbeta_h(\bv_h)})^*,\Omega}(\varphi'_{\smash{|\Dhke
              \bv_h|}}(|\Dhke \bv_h - \BD \bv|))\,.\notag
      \end{align}
      Due to $\bbeta_h(\bu_h)\in \mathbb{R}$, Prop.~\ref{prop:n-function_E}
      is applicable and yields, also using $h_F\sim h_K$, and
      $\abs{\Dhke \bw_h}\le \abs{\Ghke \bw_h}$ for all $\bw_h \in
      W^{1,p}(\mathcal T_h)$, that
      \begin{align}
        \label{eq:max-shift.2}
        \rho_{\varphi_{\bbeta_h(\bv_h)},\Omega}(\Dhke(\BE_h  \bz_h - \bz_h)) \lesssim m_{\varphi_{\bbeta_h(\bu_h)},h}(\bz_h)\,.
      \end{align}
      By definition of the max-shift functional and Rem.~\ref{rem:phi_a}, we obtain
      $ (\varphi_{\bbeta_h(\bv_h)})^*\leq (\varphi_{\smash{\vert
          \Dhke \bv_h\vert}})^*$. Using this, the
      equivalence \eqref{eq:psi'} for
      $\psi=\varphi_{\smash{\vert \Dhke\bv_h\vert}}$, and
      \eqref{eq:hammera}, we obtain
      \begin{align}\label{eq:max-shift.3}
        \begin{aligned}
          &\rho_{(\varphi_{\bbeta_h(\bv_h)})^*,\Omega}(\varphi'_{\smash{|\Dhke
              \bv_h|}}(|\Dhke \bv_h - \BD \bv|))
          \\
          &\leq\rho_{(\varphi_{|\Dhke \bv_h|})^*,\Omega}(\varphi'_{\smash{|\Dhke
              \bv_h|}}(|\Dhke \bv_h - \BD
          \bv|))
          \\
          &\lesssim \|\BF(\Dhke\bv_h) - \BF(\BD
          \bv)\|^2_{2,\Omega}\,.
        \end{aligned}
      \end{align}
    Eventually, combining \eqref{eq:max-shift.2} and \eqref{eq:max-shift.3} with \eqref{eq:max-shift.1}, we conclude the claimed estimate.
    \end{proof}
	
    \begin{prop}[well-posedness and stability] \label{prop:stabPh} Assume that $\SSS$ satisfies
      Ass.~\ref{assum:extra_stress} for some \linebreak ${p\in (1,\infty)}$, $\delta\in
      [0,\infty)$. Then, for $\alpha>0$ sufficiently large, depending
      only  on $k$, $\omega_0$   and the characteristics of $\SSS$,
      the problem (P$_h$), i.e., \eqref{eq:primal}, \eqref{def_shift}, 
      admits a solution $\bv_h\in \Vhk(0)$ for every $\bff \in
      \Vo^*$. Moreover, there holds 
      \begin{align}\label{eq:apriPh}
        \|\bv_h\|_{\nabla ,p,h}^p \lesssim \|\bff\|_{-1,p'}^{p'}
        +\delta^p\,\vert \Omega\vert\,,
      \end{align}
      where the constant depends only on $k$, $\omega_0$, and the
      characteristics of $\BS$. 
    \end{prop}
    
    \begin{proof}
      We equip $V_h^k$ with the norm $\|\cdot\|_{\nabla, p,h}$,
      and consider the operator \linebreak 
      ${\BA_h\coloneqq \BA_h^k\colon V_h^k\to (V_h^k)^*}$, for
      every $\bv_h, \bz_h\in V_h^k $, defined via
      \begin{align*}
        \langle \BA_h\bv_h,\bz_h\rangle_{V_h^k}\coloneqq 
        (\SSS(\Dhke \bv_h),\BD \BE_h  \bz_h)_{\Omega}
        + \alpha\, \langle \SSS_{\bbeta_h(\bv_h)}( h_\Gamma^{-1} \jump{\bv_h \otimes\bn} ) , \jump{\bz_h \otimes\bn}\rangle_{\Gamma_h}\,.
      \end{align*}
      Since $V_h^k$ is finite dimensional, consists of broken
      polynomials, \cite[Lem.~3.18]{bdr-7-5}, the properties of $\SSS$
      and $\SSS_a$, $a\ge 0$, imply that the operator 
      $\BA_h\colon V_h^k\to (V_h^k)^*$, for every fixed $h>0$,
      is well-defined, continuous, and, thus, pseudo-monotone.
To prove the boundedness of $\BA_h\colon V_h^k\to (V_h^k)^*$, we use
H\"older's inequality to get
    \begin{align}\label{lem:ldg_stress.5a}
      \langle \BA_h\bv_h ,\bz_h \rangle_{\Vhk}
      & \lesssim \|\SSS(\Dhke\bv_h )\|_{p',\Omega}\|\BD \BE_h\bz_h \|_{p,\Omega}
      \\
      &\quad+\alpha  \,\big\|h_\Gamma^{1/p'}\SSS_{\bbeta ( \bv_h)}
      (h_\Gamma^{-1}\jump{\bv_h \otimes
      \bn})\big\|_{p',\Gamma_{h}}
      \big\|h_\Gamma^{-1/p'}\jump{\bz_h
      \otimes \bn}\big\|_{p,\Gamma_{h}}\notag 
      \\
      &=\vcentcolon I_1\,\|\BD \BE_h \bz_h \|_{p,\Omega}+\alpha\,I_2
        \,\|h_\Gamma^{-1/p'}\jump{\bz_h \!\otimes\!
        \bn}\|_{p,\Gamma_{h}}\,.\notag 
    \end{align}
    Using that
    $\vert\SSS(\BA)\vert \leq c\,(\delta^{p-1}+\vert\BA\vert^{p-1})$
    for all ${\BA\in \mathbb{R}^{d\times d}}$
    (cf.~Ass.~\ref{assum:extra_stress}, Rem.~\ref{rem:phi_a}),
    $\abs{\Dhke \bv_h}\le \abs{\Ghke \bv_h}$, and the norm
    equivalence \eqref{eq:eqiv0}, we get that
    \begin{align}
    \label{lem:ldg_stress.6a}
    \begin{aligned}
        I_1^{p'}&  \lesssim\|\bv_h \big\|_{\nabla,p,h}^p+\delta^p\,\vert \Omega\vert\,.
    \end{aligned}
    \end{align}
    Similarly, using that
    $\vert\SSS_a(\BA)\vert \leq
    c\,(\delta^{p-1}+a^{p-1}+\vert\BA\vert^{p-1})$ for all
    $\BA\in \mathbb{R}^{d\times d}$~and~${a\ge 0}$ (cf.~Ass.~\ref{assum:extra_stress}, Rem.~\ref{rem:phi_a}),
    $\int_{\Gamma_h} h_\Gamma \,\textup{d}s\sim |\Omega|$,
    $\abs{\Dhke \bv_h}\le \abs{\Ghke \bv_h}$,
    and the equivalence of norms on the finite dimensional space $\Vhk$
    implying that 
    $\|\Dhke \bv_h\|_{\infty,\Omega}\le \|\Dhke
    \bv_h\|_{\infty,\Omega} +\|h_\Gamma^{-1}\jump{\bv_h
      \otimes \bn}\|_{\infty,\Gamma_{h}}\le 
    c(\Vhk)\,\|\bv_h\|_{\nabla,p,h}$, 
we obtain
    \begin{align}
    \label{lem:ldg_stress.7a}
    \begin{aligned}
        I_2^{p'}&\lesssim \|h_\Gamma^{-1/p'}\jump{\bv_h \otimes \bn}\|_{p,\Gamma_{h}}^p
        + (\delta^p + \bbeta_h(\bv_h)^p)\int_{\Gamma_h} h_\Gamma \,\textup{d}s
        \\
        &\lesssim
        \begin{cases}
         \|\bv_h\|_{\nabla,p,h}^p+\delta^p\,\vert
        \Omega\vert & \text{if }p\leq 2\,,
        \\
        (1+ c(\Vhk)^p\,\abs{\Omega})\, \|\bv_h\|_{\nabla,p,h}^p +\delta^p\,\vert
        \Omega\vert & \text{if }p>2\,.
          \end{cases}
        \end{aligned}
    \end{align}
    Finally, using \eqref{lem:ldg_stress.6a}, \eqref{lem:ldg_stress.7a}
    and the stability estimate \eqref{eq:stab-orlicz1} for $\psi(t)=t^p$ in
    \eqref{lem:ldg_stress.5a}, we~conclude~that
    \begin{align}\label{lem:ldg_stress.7.1a}
      \big\|\BA_h\bv_h \big\|_{\smash{(\Vhk)^*}}^{p'}\lesssim
      \begin{cases}
         \|\bv_h\|_{\nabla,p,h}^p+\delta^p\,\vert
        \Omega\vert & \text{if }p\leq 2\,,
        \\
        (1+ c(\Vhk)^p\,\abs{\Omega})\, \|\bv_h\|_{\nabla,p,h}^p +\delta^p\,\vert
        \Omega\vert & \text{if }p>2\,.
      \end{cases}
    \end{align}
      It remains to prove the coercivity of the operator  $\BA_h\colon V_h^k\to
      (V_h^k)^*$.  For every ${\bv_h\in V_h^k} $, using
      first that $\BD \BE_h \bv_h=\Dhke \BE_h \bv_h$,
      then that $\SSS (\BA):\BA\sim \varphi(\vert \BA\vert )$ and
      ${\SSS_{\bbeta_h(\bv_h)} (\BA):\BA\sim
        \varphi_{\bbeta_h(\bv_h)}(\vert \BA\vert )}$
      for~all~${\BA\in \mathbb{R}^{d\times d}}$ (cf.\
    Prop.~\ref{lem:hammer} and Rem.~\ref{rem:sa}),
      $\varphi_{\bbeta_h(\bv_h)}\ge \varphi$ (cf.\ Rem.~\ref{rem:phi_a}), the $\varepsilon$-Young inequality \eqref{ineq:young}
      with $\psi=\varphi$, and Prop.~\ref{prop:n-function_E} with
      $\psi=\varphi$, we find that
      \begin{align*}
        \langle \BA_h\bv_h,\bv_h\rangle_{V_h^k}
        &=(\SSS(\Dhke\bv_h),\Dhke
          \bv_h)_{\Omega}+ \alpha\, \langle \SSS_{\bbeta_h(\bv_h)}(
          h_\Gamma^{-1} \jump{ \bv_h \otimes\bn} ) , \jump{ \bv_h
          \otimes\bn}\rangle_{\Gamma_h} 
        \\
        &\quad + (\SSS(\Dhke \bv_h),\Dhke (\BE_h \bv_h- \bv_h))_{\Omega}
        \\
        &	\ge (c-\varepsilon)\,
          \rho_{\varphi,\Omega}(\Dhke \bv_h)+\alpha \,c\,
          m_{\varphi_{\bbeta_h(\bv_h)},h}(\bv_h)
        \\
        &\quad -c_\varepsilon\,
          \rho_{\varphi,\Omega}(\Dhke (\BE_h \bv_h- \bv_h)) 
        \\
        &	\ge (c-\varepsilon)\,
          \rho_{\varphi,\Omega}(\Dhke \bv_h)+(\alpha
          \,c-c_\varepsilon)\, m_{\varphi,h}(\bv_h)\,. 
      \end{align*}
      Consequently, choosing first $\varepsilon>0$ sufficiently small
      and, subsequently, $\alpha>0$ sufficiently large, also using
      $\varphi(t)+t^p\sim t^p+ \delta^p$ for all $t\ge 0$ (cf.~\cite{bdr-7-5}),
      $\int_{\Gamma_h}h_\Gamma \, \textup{d}s\sim |\Omega|$,
      the norm equivalence \eqref{eq:equi2}, and Korn's inequality
      Prop.~\ref{korn}, we arrive, for  every $\bv_h\in V_h^k $, at
      \begin{align}\label{eq:apri}
        \begin{aligned}
          \langle \BA_h\bv_h,\bv_h\rangle_{V_h^k} &\gtrsim \|\Dhke
          \bv_h\|_{p,\Omega}^p+\alpha
          \|h_{\Gamma}^{\smash{-1/p'}}\jump{\bv_h\otimes
            \bn}\|_{p,\Gamma_h}^p-\delta^p\,\vert
          \Omega\vert\,(1+\alpha)\\&\gtrsim \|
          \bv_h\|_{\nabla ,p,h }^p-\delta^p\,\vert
          \Omega\vert \,.
        \end{aligned}
      \end{align}
      Putting everything together, we proved that the operator
      $\BA_h\colon V_h^k\to (V_h^k)^*$ is well-defined, bounded,
      pseudo-monotone, and coercive.  Since $V_h^k(0)$ is a
      non-trivial subspace of $\Vhk$ (cf.~\cite{KVZ.2021}) this
      implies that $\BA_h$ viewed as an operator
      $\BA_h\colon V_h^k(0)\to (V_h^k(0))^*$ is also well-defined,
      bounded, pseudo-monotone, and coercive and, thus, surjective
      (cf.~\cite[Thm.\ 27.A]{zei-IIB}). The apriori estimate
      \eqref{eq:apriPh} follows
      immediately from \eqref{eq:apri}, if we use $\bz_h =\bv_h $ in
      \eqref{eq:primal}, and the stability of $\BE_h$ in
      Prop.~\ref{prop:n-function_E} for $\psi (t)=t^p$.
    \end{proof}

    \begin{prop}[well-posedness and
      stability] \label{prop:stabqh} Assume that $\SSS$ satisfies
      Ass.~\ref{assum:extra_stress} for some \linebreak$p\in (1,\infty)$,
      $\delta\in [0,\infty)$. Then, for $\alpha>0$ sufficiently large,
      depending only on $k$, $\omega_0$ and the characteristics of
      $\SSS$, the problem (Q$_h$), i.e.,
      \eqref{eq:primal0.1}--\eqref{def_shift}, admits a solution
      $(\bv_h, q_h)\in \Vhk \times \Qhkeo$ for every $\bff \in
      \Vo^*$. Moreover, there holds 
      \begin{align}\label{eq:apriQh}
        \begin{aligned}
          \|\bv_h\|_{\nabla ,p,h}^p &\lesssim
          \|\bff\|_{-1,p'}^{p'} +\delta^p\,\vert \Omega\vert\,,
          \\
          \|q_h\|^{p'}_{p',\Omega} &\lesssim
          \begin{cases}
         \|\bff \|_{-1,p'}^{p'}+\delta^p\,\vert
        \Omega\vert & \text{if }p\leq 2\,,
        \\
        (1+ c(\Vhk)^p\,\abs{\Omega})\, \|\bff \|_{-1,p'}^{p'} +\delta^p\,\vert
        \Omega\vert & \text{if }p>2\,,
      \end{cases}
        \end{aligned}
      \end{align}
      where the constants depend only on $k$, $\omega_0$, and the
      characteristics of $\BS$. 
    \end{prop}
	
\begin{proof}
  In \cite[Sec.~4.4]{JLMNR17} it is shown that there exists a constant
  $\beta>0$ such that for all $q_h \in \Qhkeo$ there holds the LBB
  condition for $p=2$
  \begin{align*}
    \beta\|q_{h}\|_{2,\Omega}\leq \sup_{\bw_{h}\in
    \Vhk\setminus\{\mathbf{0}\}}{\frac{(\Divhke\bw_{h},q_h)_\Omega}{\|\bv_{h}\|_{\nabla,2,h}}}
    =\sup_{\bw_{h}\in
    \Vhk\setminus\{\mathbf{0}\}}{\frac{-(\bw_{h},\smash{\mathcal
    {G}_{h}^k} q_h)_\Omega}{\|\bv_{h}\|_{\nabla,2,h}}} \,,
  \end{align*}
  where we used \eqref{eq:div-dg.0a}. This and \cite[Lem.\ 12.1]{EG21}
  yield for all $q_h \in \Qhkeo$ that 
  \begin{align}\label{eq:lbb}
    \beta\|q_{h}\|_{p',\Omega}\leq \sup_{\bw_{h}\in
    \Vhk\setminus\{\mathbf{0}\}}{\frac{(\Divhke\bw_{h},q_h)_\Omega}{\|\bv_{h}\|_{\nabla,p,h}}}
    =\vcentcolon\big\|\smash{\mathcal {G}_{h}^k} q_{h}\big\|_{\smash{(\Vhk)^*}}\,.
  \end{align}
  A direct consequence of the LBB condition \eqref{eq:lbb} is the
  surjectivity and, thus, bijectivity of the injective DG gradient
  operator
  $\smash{\mathcal{G}_{h}^k} \colon \Qhkeo\to(\Vhk(0))^\circ\! $,
  defined via \eqref{eq:div-dg.0a}, where
  $$\smash{(\Vhk(0))^\circ \coloneqq \{\bff_{h}\in (\Vhk)^*\mid\langle
  \bff_{h},\bz_{h}\rangle_{\Vhk}= 0\textup{ for all }\bz_{h}\in
  \Vhk(0)\}},$$
  denotes the annihilator of $\Vhk(0)$. It follows from
  Prop.~\ref{prop:n-function_E} that the adjoint operator ${\BE_h^*\colon\!
  \Vo^* \!\to\! (\Vhk) ^*}$ is well-defined and bounded. Moreover, \eqref
  {eq:primal} and the definition of $\BA_h\colon \Vhk \to (\Vhk)^*$
  yield that for all $\bz_h \in V_h^k(0)$ there holds 
  \begin{align*}
    \langle \BA_h\bv_h- \BE_h^*\bff,\bz_h\rangle_{V_h^k}=0\,,
  \end{align*}
  i.e.,  $\BA_{h}\bv_{h}-\BE_h^*\bff \in (\Vhk(0))^\circ$. Thus,  we conclude the
  existence of $q_{h}\in \Qhkeo$ such that 
  \begin{align}
    \label{eq:well_posed}
    \smash{\mathcal {G}_{h}^k q_{h}=\BA_{h}\bv_{h}-\BE_h^*\bff\quad\text{ in }(\Vhk)^*}\,,
  \end{align}
  i.e., $\smash{(\bv_{h},q_{h})^\top\in \Vhk\times \Qhkeo}$ solves
  \eqref{eq:primal0.1}--\eqref{def_shift}.  Apart from that, appealing
  to \eqref{eq:lbb} and \eqref{eq:well_posed}, we have, also using
  the boundedness of $\BE_h^*$, \eqref {lem:ldg_stress.7.1a}, \eqref{eq:apriPh} and Young's inequality, that
  \begin{align*}
    \beta \|q_{h_n}\|_{p'}^{p'}
    &\leq \big\|\mathcal{G}^k_{h} q_{h}
      \big\|^{p'}_{\smash{(\Vhk)^*}} 
    =\big\|\BA_h \bv_{h}-\BE_h^*\bff \big\|^{p'}_{\smash{(\Vhk)^*}}
    \\
    &\lesssim
      \begin{cases}
         \|\bv_h\|_{\nabla,p,h}^p+\delta^p\,\vert
        \Omega\vert  + \|\bff \|_{-1,p'}^{p'}& \text{if }p\leq 2
        \\
        (1+ c(\Vhk)^p\,\abs{\Omega})\, \|\bv_h\|_{\nabla,p,h}^p +\delta^p\,\vert
        \Omega\vert  +\|\bff \|_{-1,p'}^{p'}& \text{if }p>2\
      \end{cases}
    \\
    &\lesssim
      \begin{cases}
         \|\bff \|_{-1,p'}^{p'}+\delta^p\,\vert
        \Omega\vert & \text{if }p\leq 2\,,
        \\
        (1+ c(\Vhk)^p\,\abs{\Omega})\, \|\bff \|_{-1,p'}^{p'} +\delta^p\,\vert
        \Omega\vert & \text{if }p>2\,.
      \end{cases}
  \end{align*}
which yields \eqref{eq:apriQh}$_2$.
\end{proof}

\begin{rem}\label{rem:modular}
{\rm	
The proof of the estimate \eqref{eq:apriQh}$_2$ is based on the norm version of the LBB condition. Proceeding as in \cite[Sec.~4]{bdr-phi-stokes} on could use a modular version of the LBB condition. Thus, it seems plausible that our treatment of the $p$-Stokes problem can be extended to the $\phi$-Stokes problem for balanced N-functions $\phi$ (cf.~\cite{br-multiple-approx}, \cite{DHKZ.2025}).
}
\end{rem}
From \eqref{eq:primal0.1} and \eqref{eq:q1}, it follows that the error
equation, for every $\bz_h\in \Vhk$, takes the form
\begin{align}\label{eq:error_equation_primal}
  \begin{aligned}
    &(\SSS(\Dhke \bv_h) - \SSS(\BD \bv),\BD \BE_h
    \bz_h)_{\Omega} + \alpha\, \langle
    \SSS_{\bbeta_h(\bv_h)}(h_\Gamma^{-1} \jump{\bv_h\otimes \bn}),
    \jump{\bz_h \otimes \bn}\rangle_{\Gamma_h}
    \\
    &\quad -\hskp{q_h-q}{\Divhke \bz_h}_\Omega = 0\,,
  \end{aligned}
\end{align}
where we 	also used that $\Divhke \bz_h =\divo \BE_h\bz_h$ for every
$\bz_h \in \Vhk$ (cf.~\eqref{eq:divo}).
	
\begin{theorem}\label{thm:error_primal}
  Let $(\bv,q)  \!\in\! \Vo \times \Qo$ be a solution of \eqref{eq:q1}
  and ${(\bv_h, q_h)\!\in\! \Vhk \times \Qhkeo}$ be a solution of
  \eqref{eq:primal0.1}--\eqref{def_shift}. If $\alpha >0$ is sufficiently large,
  we have that
  \begin{align}\label{eq:opt}
    \begin{aligned}
      &\|\BF(\Dhke \bv_h) - \BF(\BD \bv)\|^2_{2,\Omega}
      + m_{\varphi_{\bbeta_h(\bv_h)},h}(\bv_h) \\
      &\lesssim \inf_{\bw_h \in \Vhk(0)} \big(\|\BF(\Dhke \bw_h) -
      \BF(\BD \bv)\|^2_{2,\Omega} +
      m_{\varphi_{\bbeta_h(\bv_h)},h}(\bw_h)\big)\,,
    \end{aligned}
  \end{align}
  where the constant depends only on $k$, $\omega_0$, and the characteristics of $\SSS$.
\end{theorem}

\begin{proof}
  Adding and subtracting $\Dhke \bz_h\!\in\! \Xhks$ in
  \eqref{eq:error_equation_primal} for arbitrary ${\bz_h\in \Vhk}$,
  using that $\Dhke \BE_h \bz_h = \BD \BE_h \bz_h$, we get
  \begin{align}\label{eq:error_primal.1}
    \begin{aligned}
      0 &= (\SSS(\Dhke \bv_h) - \SSS(\BD \bv), \Dhke\bz_h)_{\Omega}
+ \alpha\, \langle \SSS_{\bbeta_h(\bv_h)}(h_\Gamma^{-1}
      \jump{\bv_h\otimes \bn}), \jump{\bz_h \otimes
        \bn}\rangle_{\Gamma_h}
      \\
      &\quad+ (\SSS(\Dhke \bv_h) - \SSS(\BD \bv), \Dhke(\BE_h \bz_h
      - \bz_h))_{\Omega}
-\hskp{q_h-q}{\Divhke \bz_h}_\Omega
      \\
      & \eqqcolon \mathfrak{I}_1 + \alpha\, \mathfrak{I}_2
      + \mathfrak{I}_3+\mathfrak{I}_4\,.
    \end{aligned}
  \end{align}
  Next, we choose $\bz_h = \bv_h - \bw_h\in \Vhk(0)$, where
  $\bw_h\in \Vhk(0)$ is arbitrary, and estimate $\mathfrak{I}_1$,
  $\mathfrak{I}_2$, $\mathfrak{I}_3$, $\mathfrak{I}_4$:
		
  \textit{ad $\mathfrak{I}_1$.} Using \eqref{eq:hammera}, and
  \eqref{eq:hammere}, the $\varepsilon$-Young inequality
  \eqref{ineq:young} with $\psi=\varphi_{\vert {\BD} \bv\vert}$, we
  find that
  \begin{align}
    \label{eq:error_primal.2}
      \mathfrak{I}_1
      &=(\SSS(\Dhke \bv_h) - \SSS(\BD \bv), \Dhke \bv_h - \BD \bv +
        \BD \bv -\Dhke \bw_h)_{\Omega}
    \\
      &\geq c\, \|\BF(\Dhke\bv_h) - \BF(\BD \bv)\|^2_{2,\Omega} -
      \vert (\SSS(\Dhke\bv_h) - \SSS(\BD \bv),\BD \bv -\Dhke
        \bw_h)_{\Omega}      \vert \notag
    \\
      &\geq (c-\varepsilon)\, \|\BF(\Dhke \bv_h) - \BF(\BD
      \bv)\|^2_{2,\Omega} - c_\varepsilon \, \|\BF(\BD \bv) -
      \BF(\Dhke \bw_h)\|^2_{2,\Omega}\,.\notag 
  \end{align}
		
  \textit{ad $\mathfrak{I}_2$.} Using that
  $\SSS_{\bbeta_h(\bv_h)}(\BA): \BA \sim
  \varphi_{\bbeta_h(\bv_h)}(\BA)$ and $\abs{\SSS_{\bbeta_h(\bv_h)}(\BA)}\sim 
  \varphi_{\bbeta_h(\bv_h)}'(\BA)$ holds uniformly in
  $\BA\in \mathbb{R}^{n\times d}$ (cf.~Prop.~\ref{lem:hammer}
  and Rem.~\ref{rem:sa}) and the $\varepsilon$-Young inequality
  \eqref{ineq:young} with $\psi=\varphi_{\bbeta_h(\bv_h)}$,
  we obtain
  \begin{align}
    \label{eq:error_primal.3}
    \begin{aligned}
      \mathfrak{I}_2 &=
      \langle \SSS_{\bbeta_h(\bv_h)}(h_\Gamma^{-1}
      \jump{\bv_h\otimes \bn}),\jump{\bv_h\otimes \bn} -
      \jump{\bw_h\otimes \bn}\rangle_{\Gamma_h}
      \\ 
      &\geq c\, m_{\varphi_{\bbeta_h(\bv_h)},h} (\bv_h) -
       \abs{\langle h_\Gamma\,
      \varphi_{\bbeta_h(\bv_h)}'(h_\Gamma^{-1} \jump{\bv_h\otimes
        \bn}),h_\Gamma^{-1} \jump{\bw_h\otimes \bn}\rangle_{\Gamma_h}}
      \\ 
      &\geq (c- \varepsilon )\,
      m_{\varphi_{\bbeta_h(\bv_h)},h} (\bv_h) - 
      c_\varepsilon\, m_{\varphi_{\bbeta_h(\bv_h)},h} (\bw_h)\,.
    \end{aligned}
  \end{align}
		
  \textit{ad $\mathfrak{I}_3$.} Using Lem.~\ref{lem:max-shift}, we
  find that
  \begin{align}
    \label{eq:error_primal.5}
    \begin{aligned}
      \mathfrak{I}_3 &\leq \kappa\, \|\BF(\Dhke\bv_h) - \BF(\BD
      \bv)\|^2_{2,\Omega} +
      c_\kappa\, m_{\varphi_{\bbeta_h(\bv_h)},h}(\bz_h)\\
      &\lesssim \kappa \,\|\BF(\Dhke \bv_h) - \BF(\BD
      \bv)\|^2_{2,\Omega} + c_\kappa\,
      (m_{\varphi_{\bbeta_h(\bv_h)},h}(\bv_h) +
      m_{\varphi_{\bbeta_h(\bv_h)},h}(\bw_h))\,.
    \end{aligned}
  \end{align}

  \textit{ad $\mathfrak{I}_4$.} From \eqref{eq:divo} follows
  \begin{align}
    \label{eq:error_primal.5a}
    \begin{aligned}
      \mathfrak{I}_4 &=0\,.
    \end{aligned}
  \end{align}
  Combining \eqref{eq:error_primal.2}, \eqref{eq:error_primal.3},
  \eqref{eq:error_primal.5}, and \eqref{eq:error_primal.5a}, for
  $\varepsilon,\kappa>0$ sufficiently small and $\alpha>0$
  sufficiently large, we arrive at the claimed best-approximation
  result.
\end{proof}

Now we would like to derive from this quasi-optimal, pressure-robust
error estimate in Thm.~\ref{thm:error_primal} the convergence of the
velocities $\bv_h$ of the problem (Q$_h$) and (P$_h$), resp., to the solution $\bv$
of the problem (P) and (Q), resp., and convergence rates under
additional assumptions on the regularity of $\bv$. In
\cite{bgkr-optimal-dg}, where the $p$-Laplacian problem was treated,
we achieved this with the help of the stability and approximation
properties of the Scott--Zhang quasi-interpolation operator
${\Pi^{\textup{SZ}}\colon V\to \Vhk \cap V}$
(cf.~\cite{zhang-scott}). This has the advantage that the jump term in
the analogue of Thm.~\ref{thm:error_primal} \mbox{\cite[Thm.~4.6]{bgkr-optimal-dg}} vanishes.  Due to the divergence constraint
this strategy is not applicable here. There exist interpolation
operators into conforming Finite Elements, preserving the divergence,
however, only in a $(\Qhk)^*$-sense. Thus, we will work with the
(local) $L^2$--projection into $V_h^k(0)$ instead. This projection has
similar stability and approximation properties (cf.~Sec.~\ref{sec:aux}) as the Scott--Zhang
quasi-interpolation operator and maps into $\Vhk(0)$.

We denote by ${\PiDG:L^1(\Omega)^d\to V_h^k(0)}$, the \textit{(local)
$L^2$-projection} into $V_h^k(0)$, $k \in \setN_0$, which for every
$\bv \in \smash{L^1(\Omega)^d}$ and $\bz_h \in V_h^k(0)$
is~defined~via
\begin{align}
  \label{eq:PiDG}
  \bighskp{\PiDG \bv}{\bz_h}_\Omega=\hskp{\bv}{\bz_h}_\Omega\,.
\end{align}

\begin{cor}\label{cor:conv}
  For $p \le 2$ and $\alpha >0$ sufficiently large, it holds that
  \begin{align*}
    \|\BF(\Dhke \bv_h) - \BF(\BD \bv)\|^2_{2,\Omega}
    + m_{\varphi,h}(\bv_h)\to 0\quad (h\to 0)\,.
  \end{align*}
\end{cor}
\begin{proof}
  We choose $\bw_h = \PiDG \bv$ in \eqref{eq:opt}.
  From \eqref{eq:hammera}, Rem.~\ref{rem:phi_a}, $p\le 2$, that
  ${\BD \bv = \Dhke \bv}$ since $\Rhke \bv =\bzero$, ${\abs{\BD_h
    \bw} \le \abs{\nabla_h \bw}}$, $\bw \in W^{1,p}(\mathcal T_h)$, follows 
  that
  \begin{align*}
    \|\BF(\BD \bv) - \BF(\Dhke\PiDG \bv)\|^2_{2,\Omega} &\lesssim
    \|\BD \bv - \Dhke\PiDG \bv\|_{p,\Omega}^{p}
    \\
    &\lesssim \|\nabla_h ( \bv - \PiDG \bv)\|_{p,\Omega}^{p} +
    \|\Rhke(\bv -\PiDG \bv)\|_{p,\Omega}^{p}
  \end{align*}
  with a constant depending only on the characteristics of
  $\SSS$. Since $\bv \in \Vo_\divo$ we get from the approximation
  properties of $\PiDG$ 
  \eqref{eq:grad}, \eqref{eq:R} with $\psi(t)=t^p$ that
  \begin{align}\label{cor:primal_convergence.2}
    \|\BF(\BD \bv) - \BF(\Dhke\PiDG \bv)\|^2_{2,\Omega}  \to 0\quad (h\to 0)\,.
  \end{align}
  Moreover, \eqref{eq:mod} for $\psi =\phi$, and $\jump{ \bv\otimes \bn} =\bzero$ yield
  \begin{align*}
    m_{\varphi,h}(\PiDG \bv) =m_{\varphi,h}(\bv-\PiDG \bv)\to 0\,,
  \end{align*}
  which together with \eqref{cor:primal_convergence.2} is the claimed
  convergence under minimal regularity assumptions. 
\end{proof}

\begin{cor}\label{cor:rate}
  Assume that $p \le 2$, $\delta>0$ and that $\BF(\BD \bv) \in
  W^{1,2}(\Omega)^{d\times d}$. Then, for $\alpha >0$ sufficiently large, it holds that
  \begin{align*}
    \|\BF(\Dhke \bv_h) - \BF(\BD \bv)\|^2_{2,\Omega}
    + m_{\varphi,h}(\bv_h)\lesssim h^p
  \end{align*}
  with a constant  depending only on $\|\BF(\BD\bv)\|_{1,2,\Omega}$, $k$, $\omega_0$, the
  characteristics of $\SSS$, and $\delta^p\,\abs{\Omega}$.
\end{cor}
\begin{proof}
  Again, we choose $\bw_h = \PiDG \bv$ in \eqref{eq:opt}. In view of
  Cor.~\ref{cor:app_VG} we have
    \begin{align}\label{eq:h2}
    \|\BF(\Dhke \bv_h) - \BF(\BD \bv)\|^2_{2,\Omega} \lesssim h^2
  \end{align}
  with a constant  depending only on  $k$, $\omega_0$, the
  characteristics of $\SSS$, and $\|\nabla \BF(\BD\bv)\|_{2,\Omega}$. Note that this also holds for
  $\delta=0$. For $\delta> 0$ it is shown in \cite[Lem.~4.4]{bdr-7-5}
  that
  \begin{align*}
    \|\nabla ^2\bv\|_{p,\Omega}^p \lesssim \|\nabla
    \BF(\BD\bv)\|_{2,\Omega}^2  + \|\BF(\BD\bv)\|^2_{2,\Omega} + \delta^p \, \abs{\Omega} 
  \end{align*}
with a constant depending only on $p$. To treat the jump
term we use $\jump{\bv \otimes \bn} =\bzero$, $\bv -\PiDG \bv =\bv -\PiDG \bv
+\PiDG (\bv -\PiDG \bv )$, the approximation and stability properties
of $\PiDG$, namely \eqref{eq:PiDGapproxmglobal} with
$\psi=\phi$ and \eqref{eq:PiDGLW1psistable}  with $\psi=\phi$, and $p\le 2$,  to obtain
\begin{align*}
  m_{\varphi,h}(\PiDG \bv) &=m_{\varphi,h}(\bv-\PiDG \bv)
  \lesssim \rho_{\phi, \Omega}(\nabla _h (\bv-\PiDG \bv)) 
  \\
  &\lesssim \rho_{\phi, \Omega}(h\,\nabla ^2 \bv) 
  \le h^p\, \|\nabla ^2\bv\|_{p,\Omega}^p \lesssim h^p
\end{align*}
with constants depending only on $k$, $\omega_0$, the
  characteristics of $\SSS$, and $\|\BF(\BD\bv)\|_{1,2,\Omega}$.  This
  and \eqref{eq:h2} yield the assertion.
\end{proof}

\begin{rem}
  The convergence rate $h^p$ in Cor.~\ref{cor:rate} is suboptimal compared to
  the results in \cite{kr-pnse-ldg-2}, where a rate $h^2$ is shown
  under an additional assumption on the regularity of the pressure. It
  seems that  Cor.~\ref{cor:rate}  can not be improved, because for
  the quasi-optimality estimate \eqref{eq:opt} we need that the max-shift is zero for $p\le 2$, while
  the approximation properties of $\PiDG$ in \eqref{eq:PiDGapprox1} for $\psi
  =\phi$ yield only $h^p$.
\end{rem}

Due to the presence of the max-shift functional in the
quasi-optimality estimate \eqref {eq:opt} the
proof of the above corollaries seems to work only for $p\le 2$, since
for $p >2$ the max-shift $\bbeta_h(\bv_h)$ is not uniformly bounded
with respect to $h>0$. Thus, neither a adaptation of
\cite[Thm.~5.8]{kr-pnse-ldg-1}, \cite[Thm.~4.8]{kr-orlicz-ldg}
(cf.~Lem.~\ref{lem:ldg_stress.7.1a}) or Cor.~\ref{cor:convn}, Cor.~\ref{cor:ratenn} seems to work.
These observations indicate that it is currently not clear what is the
correct notion of quasi-optimality for DG approximations of the $p$-Stokes problem.

\section{Modified LDG scheme}\label{sec:rob}

In view of the discussion at the end of Section~\ref{sec:ldg} we will
modify the problems (P$_h$) and (Q$_h$) using a different shift
functional $\widetilde{\bbeta}(\bv_h)$, which allows for a pressure
robust error estimate for the velocity, convergence under minimal
regularity assumptions, and convergence rates for the
velocity, which are optimal for linear ansatz function, under additional regularity
assumptions on the velocity $\bv$ only. The new shift functional
$\widetilde{\bbeta}(\bv_h)$ is motivated by the results in
\cite{kr-orlicz-ldg,kr-pnse-ldg-1,kr-pnse-ldg-2}.

    \textbf{Problem ($\widetilde  {\mathbf Q}_h$).} For given $k \in \setN$ and
    $\bff \in \Vo^*$,
    find $(\bv_h,q_h)\!\in V_h^k\times \Qhkeo$ such that 
     for all $(\bz_h,z_h) \in \Vhk\times\Qhkeo$, it holds
    \begin{align}
      &\begin{aligned}
        &\bighskp{\SSS(\Dhke \bv_h )}{\BD \BE_h^k \bz_h}_\Omega
        +\alpha \bigskp{\SSS_{\widetilde{\bbeta}_h(\bv_h)}(h_\Gamma^{-1}
          \jump{\bv_h \otimes \bn})}{ \jump{\bz_h \otimes
            \bn}}_{\Gamma_h}
        \\
        &\quad -\bighskp{q_h}{\Divhke \bz_h}_\Omega =  \langle \bff, \BE_h^k \bz_h\rangle _{\Vo} \,,
        \label{eq:primal0.1n} 
      \end{aligned}
          \\
      \label{eq:primal0.2n}&\bighskp{\Divhke \bv_h}{z_h}_\Omega=0\,,
    \end{align}
    where $\alpha\hspace*{-0.15em}>\hspace*{-0.15em}0$ is the
    \textit{stabilisation parameter},
    $\widetilde{\bbeta}_h\colon
    \hspace*{-0.15em}V_h^k\hspace*{-0.15em}\to\hspace*{-0.15em}
    \mathbb{R}^{\ge 0}$ is the \textit{shift
      functional}, for \mbox{every}~${\bv_h\hspace*{-0.15em}\in\hspace*{-0.15em}
      V_h^k}$, defined via
    \begin{align}\label{def_shiftn}
      \widetilde{\bbeta}_h(\bv_h) \coloneqq  \abs{\avg{\PiDGe\Dhke \bv_h}}\,,
    \end{align}
    where ${\PiDGe: L^1(\Omega)^{d\times d} \to X_h^0}$ is the (local)
    $L^2$--projection into $X_h^0$, and
    $\SSS_{\widetilde \bbeta_h(\bv_h)}$ is defined in
    \eqref{eq:flux}. Note that for $k=1$ the projection $\smash{\PiDGe}$ can be
    omitted.  Again, we eliminate in the system~\eqref{eq:primal0.1n},
    \eqref{eq:primal0.2n}, the variable $q_h\in \smash{\Qhkeo}$ to derive a
    system only expressed in terms of the discrete velocity
    $\bv_h\hspace{-0.1em}\in\hspace{-0.1em} \smash{\Vhk(0)}$, which is
    the discrete counterpart of Problem~(P):
    
    \textbf{Problem ($\widetilde{\mathbf P}_h$).} For given $k \in \setN$ and
    $\bff \in \Vo^*$, find $\bv_h\in V_h^k(0)$ such
    that for all ${\bz_h\in V_h^k(0)}$, it holds
    \begin{align} \label{eq:primaln}
        \begin{aligned}
          &\hspace*{-1.5mm}\bighskp{\SSS(\Dhke \bv_h )}{\BD \BE_h^k \bz_h}_\Omega
        \!+\!\alpha \bigskp{\SSS_{\widetilde{\bbeta}_h(\bv_h)}(h_\Gamma^{-1}
          \jump{\bv_h \otimes \bn})}{ \jump{\bz_h \otimes
            \bn}}_{\Gamma_h}\!=  \langle \bff, \BE_h^k \bz_h\rangle _{\Vo} \,. \hspace*{-2mm}
        \end{aligned}
    \end{align}
    In view of the new shift functional we have to adapt
    Lem.~\ref{lem:max-shift} to the new setting.
    \begin{lem}\label{lem:shift}
      For every $\kappa>0$, there exists a constant $c_\kappa>0$,
      depending on the characteristics~of~$\SSS$, such that for every
      $\bv_h,\bz_h\in V_h^k$ and $\bv\in V$, it holds that
      \begin{align*}
        &\vert (\SSS(\Dhke \bv_h) - \SSS(\BD \bv), \Dhke(\BE_h  \bz_h
          - \bz_h))_{\Omega}\vert
        \\
        &\lesssim \kappa \,\|\BF(\Dhke \bv_h) - \BF(\BD
          \bv)\|^2_{2,\Omega}
          + \|\PiDGe \BF(\BD\bv) - \BF(\BD \bv)\|^2_{2,\Omega}
        \\
        &\quad +\sum _{F \in
          \Gamma_h} \| \mean{\BF(\BD
          \bv)}_{S_F}-\BF(\BD\bv) \|^2_{2,S_F}
 +  c_\kappa\, m_{\varphi_{\abs{\avg{\PiDGe\Dhke\bv_h}}},h}(\bz_h)\,,
      \end{align*}
      where the constants depend only on $k$, $\omega_0$, and the characteristics of $\BS$.
    \end{lem}
	
    \begin{proof}
      Using  \eqref{eq:hammere} and the $\varepsilon$-Young inequality
      \eqref{ineq:young} with $\psi=\varphi_{\abs{\PiDGe\Dhke \bv_h}}$, we find
      that
      \begin{align}
        \label{eq:max-shift.1n}
        \begin{aligned}
          &\vert (\SSS(\Dhke \bv_h) - \SSS(\BD \bv),
          \Dhke(\BE_h \bz_h - \bz_h))_{\Omega}\vert\\
          &\lesssim  c_\kappa\int_\Omega\varphi_{\abs{\PiDGe\Dhke
              \bv_h}}(\abs{\Dhke(\BE_h \bz_h - \bz_h)})\, \textup{d}x
          \\
          &\quad +\kappa\int_\Omega(\varphi_{\abs{\PiDGe\Dhke
              \bv_h}})^*(\varphi'_{\smash{|\Dhke \bv_h|}}(|\Dhke \bv_h
          - \BD \bv|))\, \textup{d}x
          \\
          &\eqqcolon c_\kappa\, I_1 +\kappa\,I_2\,,
        \end{aligned}
      \end{align}
      where $c_\kappa$ also depends on the characteristics of $\SSS$.
      Using a shift change (cf.~Lem.~\ref{lem:shift-change}),
      \eqref{eq:psi'}, again Lem.~\ref{lem:hammer}, that
      $\PiDGe \Dhke\bv_h = \mean{\Dhke\bv_h}_K$ on every
      $K\in \mathcal T_h$, \cite[Lem.~A.3]{bdr-phi-stokes}, adding and
      subtracting $\BF(\BD\bv)$ as well as $\BF(\PiDGe{\BD\bv})$, the stability of $\PiDGe$, a shift change (cf.~Lem.~\ref{lem:shift-change}), and again \cite[Lem.~A.3]{bdr-phi-stokes}, we obtain
      \begin{align}
        \label{eq:I2}
        \begin{aligned}
          I_2 &\lesssim \int_\Omega (\varphi_{\abs{\Dhke
              \bv_h}})^*(\varphi'_{\smash{|\Dhke \bv_h|}}(|\Dhke \bv_h
          - \BD \bv|))\, \textup{d}x
          \\
          &\quad + \int_\Omega \abs{\BF ({\Dhke
              \bv_h})-\BF({\PiDGe\Dhke \bv_h})}^2 \, \textup{d}x
          \\
          &\lesssim \norm{\BF ({\Dhke
              \bv_h})-\BF({\BD\bv})}_{2,\Omega}^2 +\norm{\BF
            ({\Dhke \bv_h})-\BF({\PiDGe\Dhke
              \bv_h})}_{2,\Omega}^2
          \\
          &\lesssim \norm{\BF ({\Dhke
              \bv_h})-\BF({\BD\bv})}_{2,\Omega}^2 +\norm{\BF
            ({\Dhke \bv_h})-\PiDGe\BF({\Dhke
              \bv_h})}_{2,\Omega}^2
          \\
          &\lesssim \norm{\BF ({\Dhke
              \bv_h})-\BF({\BD\bv})}_{2,\Omega}^2 +\norm{\BF
            ({\BD \bv})-\PiDGe\BF({\BD \bv})}_{2,\Omega}^2
        \end{aligned}
      \end{align}
      with constants depending only on the characteristics of
      $\SSS$.
      Using again that for every $K\in \mathcal T_h$ we have 
      $\PiDGe \Dhke\bv_h = \mean{\Dhke\bv_h}_K$,  we re-write $I_1$ as follows
      \begin{align}
        \label{eq:I1}
        \begin{aligned}
          I_1&=\sum _{K\in \mathcal T_h} \int_K\varphi_{\abs{\mean{\Dhke
              \bv_h}_K}}(\abs{\Dhke(\BE_h \bz_h - \bz_h)})\, \textup{d}x
        \eqqcolon \sum _{K\in \mathcal T_h} A_K\,.
        \end{aligned}
      \end{align}
      From the approximation properties of $\BE_h$ in Prop.~\ref{prop:n-function_E}, Lem.~\ref{lem:hammer},  and a shift change
      (cf.~Lem.~\ref{lem:shift-change}) follows
      \begin{align}
        \label{eq:A}
        \begin{aligned}
          A_K&\lesssim h_F\sum_{F\in
            \Gamma_h(K)}\int_F{\varphi_{\abs{\mean{\Dhke
                  \bv_h}_K}}}(\vert h_F^{-1}\jump{\bz_h \otimes
            \bn}_F\vert)\, \textup{d}s
          \\
          &\le c_\varepsilon \, h_F\sum_{F\in
            \Gamma_h(K)}\int_F{\varphi_{\abs{\avg{\PiDGe{\Dhke
                    \bv_h}}}}}(\vert h_F^{-1}\jump{\bz_h \otimes
            \bn}_F\vert)\, \textup{d}s
          \\
          &\quad +\vep \, h_F\sum_{F\in \Gamma_h(K)}\int_F
          \varphi_{\abs{\mean{\Dhke \bv_h}_K}}( \abs{\mean{\Dhke \bv_h}_K
        -\avg{\PiDGe{\Dhke \bv_h}}})\, \textup{d}s
          \end{aligned}
      \end{align}
      with $c_\vep$ depending also on on $k$, $\omega_0$, and the
      characteristics of $\SSS$. For given $K\in \mathcal T_h$ let
      $F\in \Gamma_h(K)$ be such that $F =\partial K \cap \partial K'$
      for some $K'\in \mathcal T_h$. Using that $\mean{\Dhke \bv_h}_K$
      and $\avg{\PiDGe{\Dhke \bv_h}}$ are constant on $F$, the
    definition of  averages, the $\Delta_2$-condition of
    $\phi$, Jensen's inequality, a shift change
      (cf.~Lem.~\ref{lem:shift-change}), $K\subset
      S_F$, $\abs{S_F} \sim \abs{K}$, adding and subtracting
      $\mean{\Dhke\bv_h}_{S_F}$,  Jensen's inequality, and
      \eqref{eq:hammera}, we get
    \begin{align*}
        &\int_F \varphi_{\abs{\mean{\Dhke \bv_h}_K}}( \abs{\mean{\Dhke \bv_h}_K
            -\avg{\PiDGe{\Dhke \bv_h}}})\, \textup{d}s
          \\
          &=h_F^{d-1}\, \varphi_{\abs{\mean{\Dhke \bv_h}_K}}\big ( \abs{\tfrac
            12 (\mean{\Dhke \bv_h}_K-\mean{\Dhke \bv_h}_{K'})}\big )
          \\
          &\lesssim h_F^{d-1}\, \varphi_{\abs{\mean{\Dhke
                \bv_h}_K}}\Big (
          \abs{ \dashint_{K'} \Dhke \bv_h-\mean{\Dhke \bv_h}_{K}\,
            \textup{d}x}\Big )
          \\
          &\lesssim h_F^{-1}\,\int_{K'} \varphi_{\abs{\mean{\Dhke \bv_h}_K}}(
          \abs{\Dhke \bv_h-\mean{\Dhke \bv_h}_{K}})\,       \textup{d}x
          \\
          &\lesssim h_F^{-1}\,\int_{S_F} \varphi_{\abs{\mean{\Dhke \bv_h}_{S_F}}}(
          \abs{\Dhke \bv_h-\mean{\Dhke \bv_h}_{K}})\,       \textup{d}x
          \\
          &\quad + h_F^{-1}\,\int_{S_F} \varphi_{\abs{\mean{\Dhke
                \bv_h}_{S_F}}}\Big (
          \dashint _K \abs{\Dhke \bv_h-\mean{\Dhke
              \bv_h}_{S_F}}\,\textup{d}y\Big )\, \textup{d}x
          \\
          &\lesssim h_F^{-1}\,\int_{S_F} \varphi_{\abs{\mean{\Dhke \bv_h}_{S_F}}}(
          \abs{\Dhke \bv_h-\mean{\Dhke \bv_h}_{S_F}})\, \textup{d}x
          \\
           &\lesssim h_F^{-1}\,\int_{S_F} \abs{\BF(\Dhke
             \bv_h)-\BF(\mean{\Dhke \bv_h}_{S_F})}^2\,       \textup{d}x 
    \end{align*}
    with constants depending only on the characteristics of
    $\SSS$. Inserting this into \eqref{eq:A}, using
    \cite[Lem.~A.3]{bdr-phi-stokes}, adding and subtracting
    $\BF(\BD\bv)$ as well as $\mean{\BF(\BD\bv)}_{S_F}$, and using
    Jensen's inequality we arrive at
      \begin{align*}
          A_K&\lesssim c_\varepsilon \, h_F\hspace*{-3mm}\sum_{F\in
            \Gamma_h(K)}\int_F\varphi_{\abs{\avg{\PiDGe{\Dhke
                  \bv_h}}}}(\vert h_F^{-1}\jump{\bz_h \otimes
            \bn}_F\vert)\, \textup{d}s
          \\
          &\quad +\vep \hspace*{-3mm}\sum_{F\in
            \Gamma_h(K)}\norm{\BF(\Dhke \bv_h)-\BF(\BD \bv)}^2_{2,S_F}
          +\norm{\BF(\BD \bv)-\mean{\BF(\BD\bv)}_{S_F}}^2_{2,S_F}
      \end{align*}
    with a constant depending only on $k$, $\omega_0$, and the
    characteristics of $\SSS$. Thus, \eqref{eq:I1}  yields
      \begin{align}
        \label{eq:I1end}
        \begin{aligned}
          I_1&\lesssim c_\varepsilon \,
          m_{\varphi_{\abs{\avg{\PiDGe{\Dhke\bv_h}}}},h} (\bz_h ) 
          \\
          &\quad +\vep \sum_{F\in
            \Gamma_h}\norm{\BF(\Dhke \bv_h)-\BF(\BD \bv)}^2_{2,S_F}
          +\norm{\BF(\BD \bv)-\mean{\BF(\BD\bv)}_{S_F}}^2_{2,S_F}
        \end{aligned}
      \end{align}
       with a constant depending only on $k$, $\omega_0$, and the
    characteristics of $\SSS$. Choosing $\vep$ such that $c_\kappa
    \vep =\kappa$, we obtain the assertion from
    \eqref{eq:max-shift.1n}, \eqref{eq:I2} and \eqref{eq:I1end}. 
  \end{proof}
  
  \begin{prop}[well-posedness and
    stability] \label{prop:stabPhn} Assume that $\SSS$ satisfies
    Ass.~\ref{assum:extra_stress} for some \linebreak $p\in (1,\infty)$,
    $\delta\in [0,\infty)$. Then, for $\alpha>0$ sufficiently large,
    depending only on the characteristics of $\SSS$, the
    problem ($\widetilde P_h$), i.e., \eqref{eq:primaln},
    \eqref{def_shiftn}, admits a solution $\bv_h\in \Vhk(0)$ for every
    $\bff \in \Vo^*$. Moreover, there holds 
    \begin{align*}
      \|\bv_h\|_{\nabla ,p,h}^p \lesssim \|\bff\|_{-1,p'}^{p'}
      +\delta^p\,\vert \Omega\vert\,,
    \end{align*}
    where the constant depends only on $k$, $\omega_0$, and the
    characteristics of $\BS$. 
  \end{prop}
    
  \begin{proof}
    The proof goes along the lines of the proof of Prop.~\ref{prop:stabPh}
    with small changes due to the new shift functional. We again equip $V_h^k$ with the norm $\|\cdot\|_{\nabla, p,h}$, and
    consider the operator 
    ${\BA_h\coloneqq \BA_h^k\colon V_h^k\to (V_h^k)^*}$, for every
    $\bv_h, \bz_h\in V_h^k $, defined via
    \begin{align*}
      \langle \BA_h\bv_h,\bz_h\rangle_{V_h^k}\coloneqq 
      (\SSS(\Dhke \bv_h),\BD \BE_h  \bz_h)_{\Omega}
      + \alpha\, \langle \SSS_{\widetilde \bbeta_h(\bv_h)}( h_\Gamma^{-1} \jump{\bv_h \otimes\bn} ) , \jump{\bz_h \otimes\bn}\rangle_{\Gamma_h}\,.
    \end{align*}
    Since $V_h^k$ is finite dimensional, consists of broken
    polynomials, \cite[Lem.~3.18]{bdr-7-5}, the properties of $\SSS$
    and $\SSS_a$, $a\ge 0$, imply that the operator
    $\BA_h\colon V_h^k\to (V_h^k)^*$, for every fixed $h>0$, is
    well-defined, continuous, and, thus, pseudo-monotone. To prove the boundedness of $\BA_h\colon V_h^k\to (V_h^k)^*$, we use
H\"older's inequality to get
    \begin{align}\label{lem:ldg_stress.5}
      \langle \BA_h\bv_h ,\bz_h \rangle_{\Vhk}
      & \lesssim \|\SSS(\Dhke\bv_h )\|_{p',\Omega}\|\BD \BE_h \bz_h \|_{p,\Omega}
      \\
      &\quad+\alpha  \big\|h_\Gamma^{1/p'}\SSS_{\smash{\abs{\avg{\PiDGe\Dhke \bv_h}}
      }}(h_\Gamma^{-1}\jump{\bv_h \otimes
      \bn})\big\|_{p',\Gamma_{h}}\big\|h_\Gamma^{-1/p'}\jump{\bz_h
      \otimes \bn}\big\|_{p,\Gamma_{h}}\notag 
      \\
      &=\vcentcolon I_1\,\|\BD\BE_h\bz_h \|_{p,\Omega}+\alpha\,I_2
\|h_\Gamma^{-1/p'}\jump{\bz_h \!\otimes\!
        \bn}\|_{p,\Gamma_{h}}\,.\notag 
    \end{align}
    The term $I_1$ is estimated as in  \eqref{lem:ldg_stress.6a}
    yielding
    \begin{align}
    \label{lem:ldg_stress.6}
    \begin{aligned}
        I_1^{p'}&  \lesssim \|\bv_h \big\|_{\nabla,p,h}^p+\delta^p\,\vert \Omega\vert\,.
    \end{aligned}
    \end{align}
    Using that
    $\vert\SSS_a(\BA)\vert \leq
    c\,(\delta^{p-1}+a^{p-1}+\vert\BA\vert^{p-1})$ for all
    $\BA\in \mathbb{R}^{d\times d}$~and~${a\ge 0}$ (cf.~Ass.~\ref{assum:extra_stress},  Rem.~\ref{rem:phi_a}),
    $\int_{\Gamma_h}h_\Gamma \, \textup{d}s\sim |\Omega|$,
    the trace inequality \eqref{eq:PiDGapproxmglobal1}, the stability
    of $\PiDGe$ (cf.~\cite[(A.11)]{dkrt-ldg}), and
    $\abs{\Dhke \bv_h}\le \abs{\Ghke \bv_h}$ 
    we obtain
    \begin{align}
    \label{lem:ldg_stress.7}
    \begin{aligned}
        I_2^{p'}&\lesssim \|h_\Gamma^{-1/p'}\jump{\bv_h \otimes \bn}\|_{p,\Gamma_{h}}^p
        + \int_{\Gamma_h} h_\Gamma \big (\delta^p +
        |\avg{\PiDGe\Dhke\bv_h}|^p\big ) \, \textup{d}s
        \\
        &\lesssim \big\|h_\Gamma^{-1/p'}\jump{\bv_h \otimes \bn}\big\|_{p,\Gamma_{h}}^p
        +\delta^p\,\vert \Omega\vert+\|\PiDGe\Dhke \bv_h\|_{p,\Omega}^p
        \\&\lesssim \|h_\Gamma^{-1/p'}\jump{\bv_h \otimes
          \bn}\|_{p,\Gamma_{h}}^p +\delta^p\,\vert \Omega\vert+\|\Ghke\bv_h \big\|_{p,\Omega}^p\,.
        \end{aligned}
    \end{align}
    Finally, using \eqref{lem:ldg_stress.6}, \eqref{lem:ldg_stress.7}, the stability estimate \eqref{eq:stab-orlicz1}, 
    and the norm equivalence \eqref{eq:eqiv0} in
    \eqref{lem:ldg_stress.5}, we~conclude~that
    \begin{align}\label{lem:ldg_stress.7.1}
      \big\|\BA_h\bv_h \big\|_{\smash{(\Vhk)^*}}^{p'}\lesssim \|\bv_h \|_{\nabla,p,h}^p+\delta^p\,\vert \Omega\vert\,.
    \end{align}
    It remains to prove the coercivity of the operator
    $\BA_h\colon V_h^k\to (V_h^k)^*$. This is essentially already
    proved in \cite[Lem.~5.2]{kr-pnse-ldg-1} based on
    \cite[Lem.~4.1]{kr-orlicz-ldg} for a slightly different shift
    functional. Due to the presence of the smoothing operator $\BE_h$
    in \eqref{eq:primaln} the argument is more subtle. Thus, the
    details are carried out in  Prop.~\ref{prop:coer}. There, it is proved
    that for sufficiently large $\alpha$ there holds
    \begin{align}\label{eq:aprin}
        \langle \BA_h\bv_h,\bv_h\rangle_{V_h^k} &\gtrsim \|
        \bv_h\|_{\nabla ,p,h }^p-\delta^p\,\vert
        \Omega\vert\,.
    \end{align}
    Now we are in the same situation as in \eqref{eq:apri}, and thus,
    we can finish the proof as in the proof of  Prop.~\ref{prop:stabPh}.
    \end{proof}
    \begin{prop}[well-posedness and
      stability] \label{prop:stabqhn} Assume that $\SSS$ satisfies
      Ass.~\ref{assum:extra_stress} for some \linebreak $p\in (1,\infty)$,
      $\delta\in [0,\infty)$. Then, for $\alpha>0$ sufficiently large,
      depending only on $k$, $\omega_0$ and the characteristics of $\SSS$, the
      problem ($\widetilde Q_h$), i.e., \eqref{eq:primal0.1n}--\eqref{def_shiftn},
      admits a solution $(\bv_h, q_h)\in \Vhk \times
      \Qhkeo$ for every $\bff \in \Vo^*$. Moreover, 
      there holds 
      \begin{align}\label{eq:apriQhn}
        \begin{aligned}
          \|\bv_h\|_{\nabla ,p,h}^p &\lesssim
          \|\bff\|_{-1,p'}^{p'} +\delta^p\,\vert \Omega\vert\,,
          \\
          \|q_h\|_{p'} ^{p'} &\lesssim
          \|\bff\|_{-1,p'}^{p'} + \delta^p\,\vert
          \Omega\vert \,,
        \end{aligned}
      \end{align}
      where the constants depend only on $k$, $\omega_0$, $p$ and the
      characteristics of $\BS$. 
    \end{prop}
    \begin{proof}
      In view of Prop.~\ref{prop:stabPhn} the proof of the existence of a
      solution $(\bv_h, q_h)\in \Vhk \times \Qhkeo$ is identical with
      the one of Prop.~\ref{prop:stabqh}. Also the proof of
      \eqref{eq:apriQhn}$_2$ works as there, just using
      \eqref{lem:ldg_stress.7.1} instead of \eqref{lem:ldg_stress.7.1a}.  
    \end{proof}

From \eqref{eq:primal0.1n} and \eqref{eq:q1}, it follows that the error
equation, for every $\bz_h\in \Vhk$, takes the form
\begin{align}\label{eq:error_equation_primaln}
  \begin{aligned}
    &(\SSS(\Dhke \bv_h)_\Omega - \SSS(\BD \bv),\BD \BE_h
    \bz_h)_{\Omega} + \alpha\, \langle
    \SSS_{\widetilde\bbeta_h(\bv_h)}(h_\Gamma^{-1} \jump{\bv_h\otimes \bn}),
    \jump{\bz_h \otimes \bn}\rangle_{\Gamma_h}
    \\
    &\quad -\hskp{q_h-q}{\Divhke \bz_h}_\Omega = 0\,,
  \end{aligned}
\end{align}
where we 	also used that $\Divhke \bz_h =\divo \BE_h\bz_h$ for every
$\bz_h \in \Vhk$ (cf.~\eqref{eq:divo}).
	
\begin{thm}\label{thm:error_primaln}
  Let $(\bv,q)  \!\in\! \Vo \times \Qo$ be a solution of \eqref{eq:q1}
  and ${(\bv_h, q_h)\!\in\! \Vhk \times \Qhkeo}$ be a solution of
  \eqref{eq:primal0.1n}. If $\alpha >0$ is sufficiently large,
  we have that
  \begin{align}\label{eq:optn}
    \begin{aligned}
      &\|\BF(\Dhke \bv_h) - \BF(\BD \bv)\|^2_{2,\Omega}
      + m_{\varphi_{\abs{\avg{\PiDGe\Dhke \bv_h}}
        }
          ,h}(\bv_h)
        \\
      &\lesssim \inf_{\bw_h \in \Vhk(0)} \big(\|\BF(\Dhke \bw_h) -
      \BF(\BD \bv)\|^2_{2,\Omega} +
      m_{\varphi_{\abs{\avg{\PiDGe\Dhke \bv_h}}
        }
          ,h}(\bw_h)\big)
      \\
    &\quad 
    +\sum_{F\in \Gamma_h} \inf_{\BQ_F\in \mathbb{R}^{d\times d}_\sym} \norm{\BF(\BD\bv)- \BQ_F}^2_{2,S_F}\,,
    \end{aligned}
  \end{align}
  where the constant depends only on $k$, $\omega_0$, and the characteristics of $\SSS$.
\end{thm}

\begin{proof}
  We proceed exactly as in the proof of Thm.~\ref{thm:error_primal}, just
  using Lem.~\ref{lem:shift} instead of Lem.~\ref{lem:max-shift}. Thus, we
  obtain for any $\bw_h \in \Vhk(0)$
  \begin{align}\label{eq:error_primal.1n}
    \begin{aligned}
      0 &= (\SSS(\Dhke \bv_h) - \SSS(\BD \bv), \Dhke\bz_h)_{\Omega} +
      \alpha\, \langle \SSS_{\widetilde \bbeta_h(\bv_h)}(h_\Gamma^{-1}
      \jump{\bv_h\otimes \bn}), \jump{\bz_h \otimes
        \bn}\rangle_{\Gamma_h}
      \\
      &\quad+ (\SSS(\Dhke \bv_h) - \SSS(\BD \bv), \Dhke(\BE_h \bz_h -
      \bz_h))_{\Omega} -\hskp{q_h-q}{\Divhke \bz_h}_\Omega
      \\
      &\ge (c-\varepsilon)\, \|\BF(\Dhke \bv_h) - \BF(\BD
      \bv)\|^2_{2,\Omega} - c_\varepsilon \, \|\BF(\BD \bv) -
      \BF(\Dhke \bw_h)\|^2_{2,\Omega}
      \\
      &\quad +\alpha\,(c- \varepsilon )\, m_{\varphi_{\bbeta_h(\bv_h)},h}
      (\bv_h) - \alpha\, c_\varepsilon\,
      m_{\varphi_{\widetilde\bbeta_h(\bv_h)},h} (\bw_h)
      \\
      &\quad -\kappa \,\|\BF(\Dhke \bv_h) - \BF(\BD
      \bv)\|^2_{2,\Omega} + c_\kappa\,
      (m_{\varphi_{\widetilde\bbeta_h(\bv_h)},h}(\bv_h) +
      m_{\varphi_{\widetilde \bbeta_h(\bv_h)},h}(\bw_h))
      \\
      &\quad -c\,\|\PiDGe \BF(\BD\bv) - \BF(\BD \bv)\|^2_{2,\Omega} -c\sum
      _{F \in \Gamma_h} \| \mean{\BF(\BD \bv)}_{S_F}-\BF(\BD\bv)
      \|^2_{2,S_F} \,.
    \end{aligned}
  \end{align}
Now, we observe that the last term can be used to absorb the penultimate term;
to see this, estimate on a given element $K\in\triang$ and any fixed facet $F\subset \partial K$:
\begin{align*}
\int_K |\BF(\BD\bv)- \PiDGe \BF(\BD\bv)|^2\, \textup{d}x
&\lesssim  
\int_{K} |\BF(\BD\bv)- \mean{\BF(\BD\bv)}_{S_F}|^2\, \textup{d}x
+ 
\int_{K} \Big |\dashint_K \BF(\BD\bv) - \mean{\BF(\BD\bv)}_{S_F}\, \textup{d}y \Big|^2\, \textup{d}x
\\&\lesssim 
2\int_{K} |\BF(\BD\bv)- \mean{\BF(\BD\bv)}_{S_F}|^2\, \textup{d}x
\lesssim 2\int_{S_F} |\BF(\BD\bv)- \mean{\BF(\BD\bv)}_{S_F}|^2\, \textup{d}x\,.
\end{align*}
Moreover, we use that $\inf_{\BQ_F\in \mathbb{R}^{d\times d}_\sym}
\int_{S_F}\abs{\BF(\BD\bv)- \BQ_F}^2 \, \textup{d}x=
\int_{S_F}\abs{\BF(\BD\bv)- \mean{\BF(\BD\bv)}_{S_F}}^2 \,
\textup{d}x$. 
To conclude the proof, we choose first $\varepsilon,\kappa>0$ sufficiently small and
  then $\alpha>0$ sufficiently large, to arrive at the claimed
  pressure-robust estimate \eqref{eq:optn}. 
\end{proof}

\begin{cor}\label{cor:convn}
  Let $\alpha$ be sufficiently large. Then, we have 
  \begin{align*}
    \|\BF(\Dhke \bv_h) - \BF(\BD \bv)\|^2_{2,\Omega}
    + m_{\varphi_{\abs{\avg{\PiDGe\Dhke  \bv_h}}    },h}(\bv_h)\to 0\quad (h\to 0)\,.
  \end{align*}
\end{cor}
\begin{proof}
  For $p \le 2$ we have that $\phi_{\widetilde\bbeta_h(\bv_h)}\le
  \phi$ (cf.~Rem.~\ref{rem:phi_a}). Thus, we choose $\bw_h = \PiDG
  \bv$ and $\BQ_F = \mean{\BF(\BD\bv)}_{S_F}$ in \eqref{eq:optn} and proceed as in the proof of
  \eqref{cor:conv} to get
  \begin{align*}
    \|\BF(\BD \bv) - \BF(\Dhke\PiDG \bv)\|^2_{2,\Omega}  \to 0\quad (h\to 0)\,,
    \\
    m_{\varphi,h}(\PiDG \bv)=m_{\varphi,h}(\bv-\PiDG \bv) \to 0\quad (h\to 0)\,.
  \end{align*}
  Moreover, Lem.~\ref{lem:conv1} implies
  \begin{align*}
    \sum_{F\in \Gamma_h} \inf_{\BQ_F\in \mathbb{R}^{d\times d}_\sym}
    \norm{\BF(\BD\bv)- \BQ_F}^2_{2,S_F} \lesssim \sum_{F\in \Gamma_h}
    \norm{\BF(\BD\bv)- \mean{\BF(\BD\bv)}}^2_{2,S_F} \to 0\quad (h\to 0)\,.
  \end{align*}
  These convergences together with \eqref{eq:optn} yields the claimed
  convergence under minimal regularity assumptions.

  For $p>2$ we choose $\bw_h \equiv \bzero $ and $\BQ_F\equiv \bzero $ in \eqref{eq:optn}. Thus,
  we have $\|\BF(\BD\bv)-\BF(\BD\bw_h)\|_{2,\Omega}^2=\|\BF(\BD\bv)\|_{2,\Omega}^2 <\infty$ in view of
  \eqref{eq:hammera}, $t^p+\delta^p \sim \phi(t) +\delta^p$ and 
  \eqref{eq:apri-cont}, as well as  
  $m_{\varphi_{\widetilde\bbeta_h(\bv_h)},h}(\bw_h)=
  \bzero$, and $\sum_{F\in \Gamma_h} \inf_{\BQ_F\in \mathbb{R}^{d\times d}_\sym}
    \norm{\BF(\BD\bv)- \BQ_F}^2_{2,S_F} = \sum_{F\in \Gamma_h} 
    \norm{\BF(\BD\bv)}^2_{2,S_F} <\infty$, due to the finite
    overlapping property of $S_F$ for $F \in \Gamma_h$.
  Consequently, the pressure-robust estimate \eqref{eq:optn} yields
  \begin{align}\label{eq:apri1}
    \begin{aligned}
      &\| \BF(\Dhke \bv_h) \|^2_{2,\Omega} + m_{\varphi_{\widetilde
          \bbeta _h(\bv_h)},h}(\bv_h)
      \\
      &\lesssim \|\BF(\Dhke \bv_h) - \BF(\BD \bv)\|^2_{2,\Omega} +
      m_{\varphi_{\widetilde \bbeta_h(\bv_h)},h}(\bv_h) +\| \BF(\BD
      \bv)\|^2_{2,\Omega}
      \\
      &\lesssim \|\BF (\BD\bv)\|_{2,\Omega}^2 <\infty\,.
    \end{aligned}
  \end{align}
  Now we choose $\bw_h = \PiDG \bv $ and $\BQ_F =
  \mean{\BF(\BD\bv)}_{S_F}$ in \eqref{eq:optn}.
  First, we observe that from \eqref{eq:hammera},  Rem.~\ref{rem:phi_a}, $p> 2$, that
  ${\BD \bv = \Dhke \bv}$ since $\Rhke \bv =\bzero$, ${\abs{\BD_h
    \bw} \le \abs{\nabla_h \bw}}$, it  follows 
  that
  \begin{align*}
    \|\BF(\BD \bv) - \BF(\Dhke\PiDG \bv)\|^2_{2,\Omega} &\lesssim
    \|\BD \bv - \Dhke\PiDG \bv\|_{p,\Omega}^{2}
    \\
    &\lesssim \|\nabla_h ( \bv - \PiDG \bv)\|_{p,\Omega}^{2} +
    \|\Rhke(\bv -\PiDG \bv)\|_{p,\Omega}^{2}
  \end{align*}
  with a constant depending only on the characteristics of
  $\SSS$. Since $\bv \in \Vo_\divo $ we get from the approximation
  properties of $\PiDG$ 
  \eqref{eq:grad}, \eqref{eq:R} with $\psi(t)=t^p$ that
  \begin{align}\label{cor:primal_convergence.2n}
    \|\BF(\BD \bv) - \BF(\Dhke\PiDG \bv)\|^2_{2,\Omega}  \to 0\quad (h\to 0)\,.
  \end{align}
  Next, we observe, using $\jump{\bv \otimes \bn}= \bzero$,
  $\PiDG\bv -\bv=\PiDG\bv -\bv -\PiDG (\PiDG\bv -\bv)$, a shift change
  in Lem.~\ref{lem:shift-change}, the approximation
  properties of $\PiDG$ in \eqref{eq:PiDGapproxmglobal}, the
  stability of $\PiDGe $ (cf.~\cite[(A.11)]{dkrt-ldg}), and
  Prop.~\ref{lem:hammer}, that for every $\vep >0$ there holds
\begin{align}
  \label{eq:conv1}
  \begin{aligned}
    m_{\varphi_{\widetilde \bbeta _h(\bv_h)},h}(\PiDG\bv -\bv)
    &\lesssim c_\vep\, m_{\varphi,h}(\PiDG\bv -\bv - \PiDG (\PiDG\bv
    -\bv))
    \\
      &\quad + \vep \int_{\Gamma_h} h_\Gamma \varphi(|\avg{\PiDGe\Dhke
        \bv_h}|)\, \textup{d}s
    \\
    &\lesssim c_\vep\, \rho_{\varphi,\Omega}(\nabla _h(\PiDG\bv -\bv))
    + \vep \, \rho_{\phi, \Omega} (\PiDGe\Dhke \bv_h)
    \\
    &\lesssim c_\vep\, \rho_{\varphi,\Omega}(\nabla _h(\PiDG\bv -\bv))
    + \vep \, \| \BF(\Dhke \bv_h) \|^2_{2,\Omega}\,. 
  \end{aligned}
\end{align}
Now, the approximation
  properties of $\PiDG$ in Lem.~\ref{lem:conv} for $\psi =\phi$ and \eqref{eq:apri1} yield for $h \to 0$
\begin{align}  \label{eq:conv1a}
  \lim_{h\to 0} m_{\varphi_{\widetilde \bbeta _h(\bv_h)},h}(\PiDG\bv
  -\bv) \lesssim \vep\,.
\end{align}
  Moreover, Lem.~\ref{lem:conv1} with $\BA= \BF(\BD\bv)$ implies
  \begin{align}  \label{eq:conv1b}
    \sum_{F\in \Gamma_h} \inf_{\BQ_F\in \mathbb{R}^{d\times d}_\sym}
    \norm{\BF(\BD\bv)- \BQ_F}^2_{2,S_F} \lesssim \sum_{F\in \Gamma_h}
    \norm{\BF(\BD\bv)- \mean{\BF(\BD\bv)}}^2_{2,S_F} \to 0\quad (h\to 0)\,.
  \end{align}
  The estimates \eqref{cor:primal_convergence.2n}--\eqref{eq:conv1b} yield the assertion also for $p>2$, since $\vep $ was
arbitrary. Thus, we proved the claimed
  convergence under minimal regularity assumptions also for $p>2$. 
\end{proof}

\begin{cor}\label{cor:ratenn}
  Assume that $\BF(\BD \bv) \in
  W^{1,2}(\Omega)^{d\times d}$. Then, for $\alpha >0$ sufficiently large, it holds that
  \begin{align*}
    \|\BF(\Dhke \bv_h) - \BF(\BD \bv)\|^2_{2,\Omega}
    + m_{\varphi_{\abs{\avg{\PiDGe\Dhke  \bv_h}}    },h}(\bv_h)
    \lesssim 
    \|h_{\mathcal T}\nabla \BF(\BD\bv)\|_{2,\Omega}^2 \lesssim h^2\,
    \|\nabla \BF(\BD\bv)\|_{2,\Omega}^2
  \end{align*}
  with a constant  depending only on 
  $k$, $\omega_0$, and the  characteristics of $\SSS$.
\end{cor}
\begin{proof}
  Again, we choose $\bw_h = \PiDG \bv$  and $\BQ_F =
  \mean{\BF(\BD\bv)}_{S_F}$ in the pressure-robust estimate
  \eqref{eq:optn}. In view of the approximation properties of $\PiDG$
  in \eqref{prop:app_V.0.3G} we have
    \begin{align*}
    \|\BF(\Dhke \bv_h) - \BF(\BD \bv)\|^2_{2,\Omega}\lesssim 
    \|h_{\mathcal T}\nabla \BF(\BD\bv)\|_{2,\Omega}^2 \lesssim h^2\, \|\nabla \BF(\BD\bv)\|_{2,\Omega}^2\,.
  \end{align*}
  To treat the jump
term we use $\jump{\bv \otimes \bn} =\bzero$, a shift change
  in Lem.~ \ref{lem:shift-change}, $\bv -\PiDG \bv =\bv -\PiDG \bv
+\PiDG (\bv -\PiDG \bv )$, \cite[Cor.~4.11,
Lem.~4.12]{kr-pnse-ldg-2}\footnote{Note that
  \cite[Lem.~4.12]{kr-pnse-ldg-2} is formulated with $\avg{\abs{\PiDGe
      \Dhke \bv_h}} $ instead of $\abs{\avg{\PiDGe \Dhke
      \bv_h}}$. However, the proofs works exactly in the same way.},
to obtain for every $\vep >0$ 
\begin{align}\label{eq:pi1}
    m_{\varphi_{\abs{\avg{\PiDGe\Dhke \bv_h}} },h}(\PiDG \bv)
    &\lesssim c_\vep \, m_{\varphi_{\abs{\BD \bv} },h}(\bv-\PiDG \bv)
    +\vep \, \int_{\Gamma_h} h_\Gamma\, \varphi_{\abs{\BD \bv} }\big(\abs
    {\BD\bv}-\abs{\avg{\PiDGe \Dhke \bv_h}}\big)\, \textup{d}s \notag 
    \\
    &\lesssim c_\vep \, \|h_{\mathcal T}\nabla
    \BF(\BD\bv)\|_{2,\Omega}^2+ \vep\, \|\BF(\BD\bv) - \BF(\PiDGe
    \Dhke \bv_h)\|_{2,\Omega}^2\,.
\end{align}
To treat the last term we add and subtract $\BF(\PiDGe\BD\bv)$, use
\cite[Lem.~A.3]{bdr-phi-stokes}, the Poincar\'e inequality on each $K
\in \mathcal T_h$, Prop.~\ref{lem:hammer}, the
  stability of $\PiDGe $ (cf.~\cite[(A.11)]{dkrt-ldg}), and a shift
  cange in Lem.~\ref{lem:shift-change} to arrive at
  \begin{align}
    \label{eq:pi0}
    \begin{aligned}
      \|\BF(\BD\bv) - \BF(\PiDGe \Dhke \bv_h)\|_{2,\Omega}^2 &\lesssim
      \|\BF(\BD\bv) - \BF(\PiDGe \BD \bv)\|_{2,\Omega}^2+
      \|\BF(\PiDGe\BD\bv) - \BF(\PiDGe \Dhke \bv_h)\|_{2,\Omega}^2
      \\
      &\lesssim \|h_{\mathcal T}\nabla \BF(\BD\bv)\|_{2,\Omega}^2 +
      \rho_{\phi_{\abs{\PiDGe\BD\bv}},\Omega}(\BD\bv -\Dhke\bv_h)
      \\
      &\lesssim \|h_{\mathcal T}\nabla \BF(\BD\bv)\|_{2,\Omega}^2 +
      \|\BF(\Dhke \bv_h) - \BF(\BD \bv)\|^2_{2,\Omega}\,.
    \end{aligned}
  \end{align}
Inserting \eqref{eq:pi0} into \eqref{eq:pi1}, choosing $\vep$ small enough the last term
is absorbed in the left-hand side of \eqref{eq:optn}. Using the
Poincar\'e inequality on each $S_F$ for $F \in \Gamma_h$ yields 
\begin{align*}
  {\sum _{F \in \Gamma_h} \|
      \mean{\BF(\BD \bv)}_{S_F}-\BF(\BD\bv) \|^2_{2,S_F}} &\lesssim 
  \|h_{\mathcal T}\nabla \BF(\BD\bv)\|_{2,\Omega}^2 \lesssim h^2\,
  \|\nabla \BF(\BD\bv)\|_{2,\Omega}^2\,. 
\end{align*}
Using the above estimates we proved the assertion. 
\end{proof}

\appendix
\section{Local $L^2$-projection}
\label{sec:aux}

In this appendix we prove stability and approximation properties of the
(local) $L^2$-projection into $V_h^k(0)$, $k \in \setN_0$, i.e., $\PiDG:L^1(\Omega)^d\to
V_h^k(0)$ is defined for every $\bv \in \smash{L^1(\Omega)^d}$ and $\bz_h
\in V_h^k(0)$ via
\begin{align*}
  \bighskp{\PiDG \bu}{\bz_h}_\Omega=\hskp{\bu}{\bz_h}_\Omega\,.
\end{align*}
All results are proven similarly to the corresponding results
for the $L^2$-projection into $V_h^k$, $k \in \setN_0$, in
\cite[App.~A.1]{dkrt-ldg} (cf.~\cite[App.~A]{kr-orlicz-ldg}). 

The projection is completely local and we could define $(\PiDG
\bg)|_K$ for $K \in \mathcal{T}_h$ via
\begin{align*}
  \hskp{\PiDG \bg}{\bz_h}_K =\hskp{\bg}{\bz_h}_K \qquad \forall\,
  \bz_h \in \Vhk(0)(K)\,.
\end{align*}
Since $\Vhk(0) \subset \Vhk$ we can use the results proved in
\cite[App.~A.1]{dkrt-ldg} (cf.~\cite[App.~A]{kr-orlicz-ldg}). In
particular, it is shown there that for $K \in \mathcal{T}_h$ and $\bg \in
L^1(\Omega)^d$ we have
\begin{align}\label{eq:infty}
  \norm{\PiDG \bg}_{L^\infty(K)} &\lesssim \dashint_{K}
  \abs{\bg}\,\textup{d}x\,.
\end{align}
This implies by Jensen's inequality that for every N-function~$\psi
\in \Delta_2$
we have for $K \in \mathcal{T}_h$ and $\bg \in
L^\psi(\Omega)^d$ that 
\begin{align}
  \label{eq:PiDGLpsistablelocal}
  \dashint_K \psi(\abs{\PiDG \bg})\,dx &\lesssim  \dashint_K \psi(\abs{\bg})\,dx
\end{align}
with a constant depending only on $k$, $\Delta _2(\psi)$ and
$\omega_0$.

To prove stability and approximation property of $\PiDG$ we introduce
the following subspaces of $L^1(\Omega)^d$ for $\ell \in \setN$
\begin{align*}
  W^{\ell,\psi}_\divo(\Omega)
  &\coloneqq \{\bw \in W^{\ell,\psi}(\Omega)^d \mid \divo
    \bw =0\}\,,
  \\
\WDGd
  &\coloneqq \{\bw_h \in W^{\ell,\psi}(\mathcal T_h)^d \mid \divo\!_h
    \bw_h =0\}\,.
\end{align*}

\begin{lem}
	\label{lem:stabpi}
	Let $\psi$ be an N-function
        satisfying the $\Delta_2$-condition.  Then, for every
        $\bw_h \in \WDGd$, $K \in \mathcal{T}_h$, and
        $0\leq j \leq \ell \leq k+1$, we have that
	\begin{align}
          \label{eq:PiDGapproxpsi}
          \dashint_K \psi\big(h_K^j \vert\nabla^j_h(\bw_h - \PiDG
          \bw_h)\vert\big)\,\textup{d}x
          &\lesssim \dashint_K \psi\big(h_K^\ell \vert \nabla^\ell
            \bw_h\vert\big)\,\textup{d}x\,,
          \\
          \dashint_K \psi\big(h_K^j \vert\nabla^j_h\PiDG
          \bw_h\vert\big)\,\textup{d}x
          &\lesssim \dashint_K \psi\big(h_K^\ell \vert \nabla^\ell
            \bw_h\vert\big)\,\textup{d}x \label{eq:PiDGstabpsi}
	\end{align}
	with  constants depending only on $\ell, k\in
        \mathbb{N}_0$, $\Delta_2(\psi)$, and $ \omega_0>0$.
\end{lem}
\begin{proof}
Since $\PiDG \bp = \bp$ for all $\bp \in \Vhk(0)$, we have for
every $\bw_h \in \WDGd $
\begin{align*}
  &\dashint_K \psi(h_K^j \abs{\nabla^j_h(\bw_h - \PiDG
      \bw_h)})\,\textup{d}x \quad
  \\
  &\lesssim \inf_{\bp \in \Vhk(0)(K)} \bigg( \dashint_K
  \psi(h_K^j \abs{\nabla^j_h(\bw_h - \bp)})\,\textup{d}x + \dashint_K \psi(h_K^j
  \abs{\nabla^j_h\PiDG(\bw_h - \bp)})\,\textup{d}x \bigg)\,.
\end{align*}
To estimate the last term we use 
$  \norm{\bp}_{\infty,K}\sim \mean{\abs{\bp}}_K $,
an inverse estimate for polynomials, Jensen's inequality
and the $L^\psi$-stability in~\eqref{eq:PiDGLpsistablelocal} to get
\begin{align*}
  \dashint_K \psi(h_K^j \abs{\nabla^j_h\PiDG(\bw_h - \bp)})\,\textup{d}x
  &\lesssim \dashint_K \psi( \abs{\bw_h - \bp})\,\textup{d}x
\end{align*}
with a constant depending only on $\ell, k\in \mathbb{N}_0$,
$\Delta_2(\psi)$, and $ \omega_0$.  Now we choose $\bp$ as the
averaged Taylor polynomial $\BQ^\ell \bw_h \in \Vhke(K)$
of~$\bw_h$. Since $\BQ^\ell$ commutes with derivatives
(cf.~\cite[Prop.~4.1.17]{BS08}) we have that
$\PiDG \BQ^\ell \bw_h =\BQ^\ell \bw_h$. Thus,
$\BQ^\ell\bw_h \in \Vhk(0)(K)$ for
$\bw_h \in \WDGd $, and we get by the
classical Sobolev--\Poincare{} estimates for N-functions
(cf.~\cite[Cor.~3.3]{dr-interpol}) the Orlicz approximability
\eqref{eq:PiDGapproxpsi}. Using the triangle inequality we also obtain
the Orlicz stability \eqref{eq:PiDGstabpsi}. 
\end{proof}
\begin{cor}\label{cor:stabpi}
  Let $\psi$ be an N-function satisfying the
  $\Delta_2$-condition. Then, for every $k\in \setN_0$,
  $\bu \in L^{\psi}(\Omega)^d$ and
  $\bw_h \in \WDGd$, we have that
  \begin{align}
    \label{eq:PiDGapprox0}
    \rho_{\psi,\Omega}\big(\bu - \PiDG \bu\big)
    &\lesssim \rho_{\psi,\Omega}(\bu)\,,
    \\
    \label{eq:PiDGapprox1a}
    \rho_{\psi,\Omega}\big(\bw_h - \PiDG \bw_h\big)
    &\lesssim  \rho_{\psi,\Omega}(h_{\mathcal T}\, {\nabla_h \bw_h})\,,
    \\
    \label{eq:PiDGapprox2}
    \rho_{\psi,\Omega}\big(\nabla_h \bw_h - \nabla_h \PiDG \bw_h\big)
    &\lesssim    \rho_{\psi,\Omega}({\nabla_h \bw_h})
  \end{align}
  with constants  depending only on $k\in\mathbb{N}_0$,
  $\Delta_2(\psi)$, and $\omega_0$. In particular, this implies for
  every $k\in \setN_0$, $\bu \in L^{\psi}(\Omega)^d$ and
  $\bw_h \in \WDGd$ that
  \begin{align}
    \label{eq:PiDGLpsistable}
    \rho_{\psi,\Omega}\big(\PiDG \bu\big)
    &\lesssim     \rho_{\psi,\Omega}(\bu)\,,
    \\
    \label{eq:PiDGLW1psistable}
    \rho_{\psi,\Omega}\big(\nabla_h \PiDG \bw_h\big)
    &\lesssim \rho_{\psi,\Omega}(\nabla_h \bw_h)
  \end{align} 
  with constants  depending only on $k\in\mathbb{N}_0$,
  $\Delta_2(\psi)$, and $\omega_0$. For every $k \in \setN$ and $\bw_h \in W^{2,\psi}_\divo(\mathcal T_h)^d$, we
  have that
  \begin{align}
    \label{eq:PiDGapprox1}
    \rho_{\psi,\Omega}(\nabla_h \bw_h - \nabla_h \PiDG
    \bw_h)
    &\lesssim \rho_{\psi,\Omega}({h_{\mathcal T}\nabla^2_h \bw_h})
  \end{align}
  with a constant depending only on $k\in\mathbb{N}_0$, 
  $\Delta_2(\psi)$, and $\omega_0$.
\end{cor}

From these stability and approximability properties of $\PiDG$ we can
derive the following estimates for our error measure in
Thm.~\ref{thm:error_primal}. We start with the volume term and the
local symmetric gradient.
\begin{prop}\label{prop:app_V}
  Let $\SSS$ satisfy Ass.~\ref{assum:extra_stress} with
  $p\in (1,\infty)$ and $\delta\in [0,\infty)$, and let $k\in
  \mathbb{N}$. Then, for every ${\bw_h\in \WDGd}$
  and $K \in \mathcal{T}_h$, it~holds
  \begin{align}
    \int_K \bigabs{\BF (\BD_h \bw_h) - \BF \big(\BD_h \PiDG \bw_h\big)}^2
    \,\mathrm{d}x
    &\lesssim  \int_K \bigabs{\BF(\BD_h \bw_h) -
      \BF\big(\mean{\BD_h \bw_h}_K\big)}^2
      \,\mathrm{d}x\,.\label{prop:app_V.0.1} 
  \end{align}
  with a constant depending only on $k$, $\omega_0$, and the
  characteristics~of~$\SSS$.  In addition, for every ${\bw_h \in
    W^{1,p}_\divo(\mathcal{T}_h)}$ with 
  ${\BF(\BD_h\bw_h)\in W^{1,2}(\mathcal{T}_h)^{d\times d}}$ and $K \in \mathcal{T}_h$, it holds
  \begin{align}
    \int_K \bigabs{\BF (\BD_h \bw_h) - \smash{\BF \big(\BD_h \PiDG \bw_h\big)}}^2
    \,\mathrm{d}x
    &\lesssim h_K^2\,\int_K \bigabs{\nabla_h\BF(\BD_h \bw_h)}^2
      \,\mathrm{d}x\,,\label{prop:app_V.0.2}
    \\
    \bignorm{\BF (\BD_h \bw_h) - \BF \big(\BD_h \PiDG
    \bw_h\big)}^2_{2,\Omega}
    &\lesssim \norm{h_{\mathcal T}\nabla_h\BF(\BD_h \bw_h)}^2_{2,\Omega} \lesssim h^2\, \norm{\nabla_h\BF(\BD_h \bw_h)}^2_{2,\Omega}\label{prop:app_V.0.3}
  \end{align}
  with  constants  depending only on  $k$, $\omega_0$, and the
  characteristics of $\SSS$.
\end{prop}
\begin{proof}
  This result is essentially proved in \cite[Thm.~5.7]{dr-interpol}
  and \cite[Thm.~3.4]{bdr-phi-stokes} in the case of a full
  gradient. The case of a symmetric gradient and the $L^2$-projection
  into $\Vhk$ is treated in \cite[Prop.~4.6]{kr-pnse-ldg-2}. This
  proof works also in our case due to the fact that for
  $\bw_h\in W^{1,p}_\divo(\mathcal{T}_h)$ the gradient of the linear polynomial $\bp_h$ 
  defined via ${\nabla \bp_h|_K = \mean {\nabla \bw_h}_K}$~for~$K \in
  \mathcal T_h$, belongs to $\Vhk(0)$. 
\end{proof}
Next we treat the jump operators.
\begin{prop}\label{prop:app_VR}
  Let $\SSS$ satisfy Ass.~\ref{assum:extra_stress} with
  $p\in (1,\infty)$ and $\delta\in [0,\infty)$, and let
  $k\in \mathbb{N}$. Then, for every
  ${\bv \in W^{1,p}_\divo(\Omega)}$ with
  ${\BF(\BD\bv)\in W^{1,2}(\Omega)^{d\times d}}$ and
  $K \in \mathcal{T}_h$, it holds
  \begin{align}
    \int_{S_F} \phi_{\abs{\BD\bv}}\big (\Rhke (\bv - \PiDG \bv)\big)
    \,\mathrm{d}x
    &\lesssim h_F^2\,\int_{S_F} \bigabs{\nabla\BF(\BD \bv)}^2
      \,\mathrm{d}x\,,\label{prop:app_V.0.2R}
    \\
    \rho_{\phi_{\abs{\BD\bv}},\Omega}\big(\Rhke ( \bv- \PiDG
    \bv)\big)
    &\lesssim \norm{h_{\mathcal T}\nabla\BF(\BD \bv)}^2_{2,\Omega}
      \lesssim h^2\, \norm{\nabla\BF(\BD
      \bv)}^2_{2,\Omega}\label{prop:app_V.0.3R} 
  \end{align}
  with  constants  depending only on  $k$, $\omega_0$, and the
  characteristics of $\SSS$.
\end{prop}
\begin{proof}
  A slightly different version of this result is proved in \cite[Prop.~4.8]{kr-pnse-ldg-2}. This
  proof works also in our case due to the fact that for
  $\bw_h\in W^{1,p}_\divo(\mathcal{T}_h)$ the gradient of the linear polynomial $\bp_h$ 
  defined via ${\nabla \bp_h|_K = \mean {\nabla \bw_h}_K}$~for~$K \in
  \mathcal T_h$, belongs to $\Vhk(0)$. 
\end{proof}
\begin{cor}\label{cor:app_VG}
  Let $\SSS$ satisfy Ass.~\ref{assum:extra_stress} with
  $p\in (1,\infty)$ and $\delta\in [0,\infty)$, and let
  $k\in \mathbb{N}$. Then, for every
  ${\bv \in W^{1,p}_\divo(\Omega)}$ with
  ${\BF(\BD\bv)\in W^{1,2}(\Omega)^{d\times d}}$ and
  $K \in \mathcal{T}_h$, it holds
  \begin{align}
    \int_K \bigabs{\BF (\BD \bv) - \smash{\BF \big(\Dhke \PiDG \bv\big)}}^2
    \,\mathrm{d}x
    &\lesssim h_K^2\,\int_K \bigabs{\nabla\BF(\BD \bv)}^2
      \,\mathrm{d}x\,,\label{prop:app_V.0.2G}
    \\
    \bignorm{\BF (\BD \bv) - \BF \big(\Dhke \PiDG
    \bv\big)}^2_{2,\Omega}
    &\lesssim \norm{h_{\mathcal T}\nabla\BF(\BD
      \bv)}^2_{2,\Omega} \lesssim h^2\, \norm{\nabla\BF(\BD
      \bv)}^2_{2,\Omega}\label{prop:app_V.0.3G} 
  \end{align}
  with  constants  depending only on  $k$, $\omega_0$, and the
  characteristics of $\SSS$.
\end{cor}
\begin{proof}
  The assertions follow immediately from Prop.~\ref{prop:app_V} and
  Prop.~\ref{prop:app_VR}, also using that $\Dhke \PiDG \bv=\BD_h\PiDG
       \bv+(\Rhke\PiDG\bv)^{\textup {sym}}$ and $\Rhke\bv
       =\mathbf{0}$ in $L^p(\Omega)$, and Prop.~\eqref{lem:hammer}.
\end{proof}

Next we treat terms on the faces.
\begin{cor}\label{cor:PiDGapproxmlocal}
  Let $\psi$ be an N-function
  satisfying the $\Delta_2$-condition and $k \in \setN_0$.  Let ${K \in \mathcal{T}_h}$ and
  $F$ be a face of~$K$. Then, for every $\bw_h \in
  \WDGd$, and $\bu_h \in \Vhk$, $k\in \setN_0$, we have
  that
  \begin{align}
    \label{eq:PiDGapproxmlocal}
     h_F\int_F \psi\big(h_F^{-1} \bigabs{\bw_h - \PiDG \bw_h}\big)\,\textup{d}s 
    &\lesssim \int_K \psi(\abs{\nabla_h \bw_h})\,\textup{d}x\,,
    \\[-1.5mm]
    \label{eq:PiDGapproxmglobal}
    m_{\psi,h}\big(\bw_h - \PiDG \bw_h\big) &\lesssim \rho_{\psi,\Omega}(\nabla_h \bw_h)\,,
    \\
    \label{eq:PiDGapproxmglobal2}
    m_{\psi,h}(\bw_h )
    &\lesssim\rho_{\psi,\Omega}\big(h_\mathcal{T}^{-1}
      \bw_h\big)+c\,\rho_{\psi,\Omega}(\nabla_h \bw_h)\,, 
    \\
    \label{eq:PiDGapproxmglobal3}
    \int_{\Gamma_h} h_\Gamma \psi(|\avg{\bw_h-\PiDG
    \bw_h}|) \, \textup{d}s&\lesssim\rho_{\psi,\Omega}( h_{\mathcal{T}}\nabla _h\bw_h)\,,
    \\
    \label{eq:PiDGapproxmglobal1}
    \int_{\Gamma_h} h_\Gamma \psi(|\avg{\bu_h}|) \, \textup{d}s
    &\lesssim\rho_{\psi,\Omega}( \bu_h)\,,
  \end{align}
  with  constants  depending only on $k$, $\Delta_2(\psi)$, and $\omega_0$.
\end{cor}
\begin{proof}
  All assertions are based on a version of the trace theorem
  \cite[Lem.~A.1]{dkrt-ldg} and proved in   \cite[App.~A.1]{dkrt-ldg}
  (cf.~\cite[Cor.~A.7]{kr-orlicz-ldg}) for the $L^2$-projection
  into $\Vhk$. In view of our analogous results for the
  $L^2$-projection into $\Vhk(0)$ above, the proofs work
  analogously. We also use that $h_F \sim h_K$ for all $F \in
  \Gamma_h$, $F \subset \partial K$.
\end{proof}

\begin{lem}\label{lem:conv}
  Let $\psi$ be an N-function
  satisfying the $\Delta_2$-condition and $k\in \mathbb{N}$. Then, for
  every $\bu \in W^{1,\psi}_\divo(\Omega)$, we have that
  \begin{alignat}{2}
    \label{eq:grad}
    \rho_{\psi,\Omega}\big(\nabla_h(\PiDG \bu -\bu)\big)&\to 0 &&\quad (h\to0)\,,
    \\
    \label{eq:mod}
    m_{\psi,\Omega}\big(\PiDG \bu -\bu\big)&\to 0 &&\quad (h\to0)\,,
    \\
    \label{eq:R}
    \rho_{\psi,\Omega}\big(\Rhk(\PiDG \bu -\bu)\big)&\to 0 &&\quad (h\to0)\,,
    \\
    \label{eq:G}
    \rho_{\psi,\Omega}\big(\Ghk(\PiDG \bu -\bu)\big)&\to 0 &&\quad (h\to0)\,,
\end{alignat}
\end{lem}

\begin{proof}
  Since for every $\bu \in W^{1,\psi}_\divo(\Omega)$, there exists a
  sequence  $(\bu^n)_{n\in \mathbb{N}} \subseteq C^{\infty}_\divo(\overline
  {\Omega})^d$ such that $ \rho_{\psi, \Omega}\big(\nabla (\bu^n-\bu)\big)\to 0 \quad  (n \to
  \infty)$ we can argue as in the proof of \cite[Lem.~A.9]{kr-orlicz-ldg}. 
\end{proof}

\begin{lem}\label{lem:conv1}
  For every $\BA\in L^{2}(\Omega)$, we have that
  \begin{alignat}{2}
    \label{eq:grad1}
 \sum_{F\in \Gamma_h} \norm{\BA- \mean {\BA}_{S_F}}^2_{2,S_F} &\to 0 &&\quad (h\to0)\,.
\end{alignat}
\end{lem}

\begin{proof}
  Since for every $\BA \in L^2(\Omega)^{d\times d}$, there exists a
  sequence  $(\BA^n)_{n\in \mathbb{N}} \subseteq C^{\infty}_0(
  {\Omega})^{d \times d}$ such that $ \|\BA^n-\BA\|_{2,\Omega}\to 0 \quad  (n \to
  \infty)$,  we  obtain, also using  Jensens inequality, 
  Poincar\'e inequality and the finite overlapping property of $S_F$
  for ${F\in \Gamma_h} $, that 
  \begin{align*}
    \sum_{F\in \Gamma_h} \norm{\BA- \mean {\BA}_{S_F}}^2_{2,S_F}
    & \lesssim \sum_{F\in \Gamma_h} \norm{\BA- \BA^n }^2_{2,S_F} + \norm
      {\BA^n- \mean {\BA^n}_{S_F}}^2_{2,S_F} +  \norm{\mean
      {\BA}_{S_F}- \mean {\BA}_{S_F}}^2_{2,S_F}
    \\
    &\lesssim \norm{\BA- \BA^n }^2_{2,\Omega} +  \sum_{F\in \Gamma_h}
      h_F^2\,\norm{\nabla \BA^n }^2_{2,S_F} \lesssim \norm{\BA- \BA^n }^2_{2,\Omega} +  
      h^2\,\norm{\nabla \BA^n }^2_{2,\Omega} \,.
  \end{align*}
  For every $\vep >0$ there exists $n_0 \in \setN$ such that
  $\norm{\BA- \BA^n }^2_{2,\Omega} \le \vep$. Using this $n_0 $ in the
  previous inequality we get
  \begin{align*}
    \lim _{h \to 0} \sum_{F\in \Gamma_h} \norm{\BA- \mean
    {\BA}_{S_F}}^2_{2,S_F} \lesssim \vep\,,
  \end{align*}
  which yields the assertion since $\vep$ was arbitrary. 
\end{proof}

\section{Coercivity}
\begin{prop}\label{prop:coer}
  Assume that $\SSS$ satisfies Ass.~\ref{assum:extra_stress} for some
  $p\in (1,\infty)$, $\delta\in [0,\infty)$. Then, for $\alpha>0$
  sufficiently large, depending only on $k$, $\omega_0$ and the
  characteristics of $\SSS$, there holds for every $\bv_h\in V_h^k $
  that 
  \begin{align*}
    \langle \BA_h\bv_h,\bv_h\rangle_{V_h^k} &\gtrsim \|
    \bv_h\|_{\nabla ,p,h }^p-\delta^p\,\vert
    \Omega\vert
  \end{align*}
  with a constant depending only on $k$, $\omega_0$ and the
  characteristics of $\SSS$.
\end{prop}
\begin{proof}
  We distinguish the cases $p \ge 2$ and $p<2$. Let us start with the
  former one ${p\ge 2}$.  For every ${\bv_h\in V_h^k} $, using 
  that $\BD \BE_h \bv_h=\Dhke \BE_h \bv_h$, 
  ${\SSS (\BA):\BA\sim \varphi(\vert \BA\vert )}$ and
  ${\SSS_{\widetilde \bbeta_h(\bv_h)} (\BA):\BA\sim
    \varphi_{\widetilde \bbeta_h(\bv_h)}(\vert \BA\vert )}$
  for~all~${\BA\in \mathbb{R}^{d\times d}}$ (cf.\
Prop.~\ref{lem:hammer} and Rem.~\ref{rem:sa}), $\varphi_{\widetilde \bbeta_h(\bv_h)}\ge \varphi$
  (cf.\ Rem.~\ref{rem:phi_a}), the $\varepsilon$-Young inequality
  \eqref{ineq:young} with $\psi=\varphi$, and the approximation
  properties of $\BE_h$ in Prop.~\ref{prop:n-function_E}
  with $\psi=\varphi$, we find that
    \begin{align*}
      \langle \BA_h\bv_h,\bv_h\rangle_{V_h^k}
      &=(\SSS(\Dhke\bv_h),\Dhke
        \bv_h)_{\Omega}+ \alpha\, \langle \SSS_{\widetilde \bbeta_h(\bv_h)}(
        h_\Gamma^{-1} \jump{ \bv_h \otimes\bn} ) , \jump{ \bv_h
        \otimes\bn}\rangle_{\Gamma_h} 
      \\
      &\quad + (\SSS(\Dhke \bv_h),\Dhke (\BE_h \bv_h- \bv_h))_{\Omega}
      \\
      &	\ge (c-\varepsilon)\,
        \rho_{\varphi,\Omega}(\Dhke \bv_h)+\alpha\,c\,  m_{\varphi,h}(\bv_h)
      -c_\varepsilon\,
        \rho_{\varphi,\Omega}(\Dhke (\BE_h \bv_h- \bv_h)) 
      \\
      &	\ge (c-\varepsilon)\,
        \rho_{\varphi,\Omega}(\Dhke \bv_h)+(\alpha
        \,c-c_\varepsilon\,c)\, m_{\varphi,h}(\bv_h)\,. 
    \end{align*}
    Consequently, choosing first $\varepsilon>0$ sufficiently small
    and, subsequently, $\alpha>0$ sufficiently large, also using
    $\varphi(t)+t^p\sim t^p+ \delta^p$ for all $t\ge 0$    (cf.~\cite{bdr-7-5}),
    $\int_{\Gamma_h}h_\Gamma \, \textup{d}s\sim |\Omega|$,
    the norm equivalence \eqref{eq:equi2}, and Korn's inequality (Prop.~\ref{korn}), we
    arrive, for every $\bv_h\in V_h^k $, at
    \begin{align*}
        \langle \BA_h\bv_h,\bv_h\rangle_{V_h^k} &\gtrsim \|\Dhke
        \bv_h\|_{p,\Omega}^p+
        \|h_{\Gamma}^{\smash{-1/p'}}\jump{\bv_h\otimes
          \bn}\|_{p,\Gamma_h}^p-\delta^p\,\vert
        \Omega\vert\,(1+\alpha)\\&\gtrsim \|
        \bv_h\|_{\nabla ,p,h }^p-\delta^p\,\vert
        \Omega\vert\,,
    \end{align*}
    which is the assertion for $p \ge 2$.

    For $p<2$ we have to be more carefull with the constants. We will
    denote by $c_i$, $i\in \setN$, concrete constants from specific
    estimates depending only on $k$, $\omega_0$ and the
    characteristics of $\SSS$, which are not changed anymore. It is shown in
    \cite[Lem.~2.10]{kang-halfspace} that there are constants $c_1,c_2$
    such that for all $a,t,s \ge 0$ and $\vep>0$ there holds
    \begin{align}
      \phi_a(t) &\ge \vep^2\, c_1\, \phi(t) - \vep^4\, \phi(a)\,,\label{eq:shiftp}
      \\
      t\, \phi'(s) &\le  c_2\,\big (\vep^{-2}\, \phi(t) +\vep ^2
                     \phi(s)\big )\,.\label{eq:youngp} 
    \end{align}
    From \eqref{eq:shiftp} with
    $t = \abs{\bv_h}, a =\abs{\avg {\PiDGe \Dhke\bv_h}}$,
    the stability properties of $\PiDG$ in \eqref{eq:PiDGapproxmglobal1}, and the stability of $\PiDGe$
    (cf.~\cite[(A.11)]{dkrt-ldg}) follows 
    \begin{align}
      \label{eq:mphi}
      m_{\varphi_{\widetilde \bbeta (\bv_h)},h}(\bv_h) \ge c_1 \, \vep
      ^2\, m_{\varphi,h}(\bv_h) -c_3\, \vep^{4}\, \rho_{\phi,
      \Omega}(\Dhke \bv_h)\,.
    \end{align}
    For every ${\bv_h\in V_h^k} $, using first that
    $\BD \BE_h \bv_h\!=\!\Dhke \BE_h \bv_h$, the equivalences
    ${\SSS (\BA):\BA\!\sim \!\varphi(\vert \BA\vert )}$ and 
    ${\SSS_{\widetilde \bbeta_h(\bv_h)} (\BA):\BA\sim
      \varphi_{\widetilde \bbeta_h(\bv_h)}(\vert \BA\vert )}$
    for~all~${\BA\in \mathbb{R}^{d\times d}}$ (cf.\
  Prop.~\ref{lem:hammer} and Rem.~\ref{rem:sa}), \eqref{eq:shiftp},
  \eqref{eq:mphi}, \eqref{eq:hammere}
    with the constant denoted by $c_7$,
    \eqref{eq:youngp} with $\vep =\kappa$, and $\phi(c_7 \, t)\le
    c_7^2\,\phi (t)$ (cf.~\cite[(2.6)]{kang-halfspace}), Prop.~\ref{prop:n-function_E}
    with $\psi=\varphi$ and constant denoted by $c_8$,  we obtain that 
    \begin{align*}
      \langle \BA_h\bv_h,\bv_h\rangle_{V_h^k}
      &\ge  c_4\, \rho_{\phi, \Omega}(\Dhke \bv_h)
        +\alpha\, c_4\, m_{\varphi_{\widetilde \bbeta
        (\bv_h)},h}(\bv_h) 
      \\
      &\quad + (\SSS(\Dhke \bv_h),\Dhke (\BE_h \bv_h- \bv_h))_{\Omega}
      \\
      &	\ge (c_4-\varepsilon^4 \, c_3\,c_4\,\alpha -\kappa^2\, c_2)\,
        \rho_{\varphi,\Omega}(\Dhke \bv_h)
      \\
      &\quad +(\vep^2\,c_1\, c_4\,\alpha
        - \kappa^{-2}\,c_2\, c_7^2\, c_8 )\, m_{\varphi,h}(\bv_h)\,.
    \end{align*}
    Now we choose for any given $\alpha$ the number $\vep$ such that $\vep ^4\, c_3\, \alpha=\tfrac 12
    $, then we choose $\kappa$ such that $\tfrac {c_4}4= \kappa^2\,
    c_2$, and finally $\alpha$ such that $\alpha^{\frac 12}
    \frac{c_1\, c_4}{(2\, c_3)^{\frac 12}}= 8\frac{c_2^2\, c_7^2 \, c_8}{c_4}$.
    Thus, we arrive at
    \begin{align*}
      \langle \BA_h\bv_h,\bv_h\rangle_{V_h^k}
      &\ge  \frac{c_4}4\, \rho_{\phi, \Omega}(\Dhke \bv_h)
        + 4\,\frac{c_2^2\, c_7^2\, c_8 }{c_4}\, m_{\varphi,h}(\bv_h)\,.
    \end{align*}
    Using $\varphi(t)+t^p\sim t^p+ \delta^p$ for all $t\ge 0$
    (cf.~\cite{bdr-7-5}),
    $\int_{\Gamma_h}h_\Gamma\, \textup{d}s \sim |\Omega|$,
    the norm equivalence \eqref{eq:equi2}, and Korn's inequality (Prop.~\ref{korn}), we
    arrive, for every $\bv_h\in V_h^k $, at
    \begin{align*}
        \langle \BA_h\bv_h,\bv_h\rangle_{V_h^k} &\gtrsim \|\Dhke
        \bv_h\|_{p,\Omega}^p+ 
        \|h_{\Gamma}^{\smash{-1/p'}}\jump{\bv_h\otimes
          \bn}\|_{p,\Gamma_h}^p-\delta^p\,\vert \Omega\vert 
      \\
       &\gtrsim \| \bv_h\|_{\nabla ,p,h }^p-\delta^p\,\vert
        \Omega\vert
    \end{align*}
    with constants depending only on $k$, $\omega_0$ and the
    characteristics of $\SSS$, which is the assertion for $p \ge 2$.
  \end{proof}

	\printbibliography

\end{document}